\documentclass[12pt]{amsart}
\usepackage{verbatim}
\usepackage{amscd}

\numberwithin{equation}{section}
\input epsf

\begin{document}

\renewcommand{\theequation}{\thesection.\arabic{equation}}
\setcounter{secnumdepth}{2}
\newtheorem{theorem}{Theorem}[section]
\newtheorem{definition}[theorem]{Definition}
\newtheorem{lemma}[theorem]{Lemma}
\newtheorem{corollary}[theorem]{Corollary}
\newtheorem{proposition}[theorem]{Proposition}
\numberwithin{equation}{section}
\theoremstyle{definition}
\newtheorem{example}[theorem]{Example}
\title[Toric generalized K$\ddot{a}$hler structures. I]
{Toric generalized K$\ddot{a}$hler structures. I}

\author[Yicao Wang]{Yicao Wang}
\address
{Department of Mathematics, Hohai University, Nanjing 210098, China}
\maketitle

\baselineskip= 20pt
\begin{abstract}
This is a sequel of \cite{Wang}, which provides a general formalism for this paper. We mainly investigate thoroughly a subclass of toric generalized K$\ddot{a}$hler manifolds of symplectic type introduced by Boulanger in \cite{Bou}. We find torus actions on such manifolds are all \emph{strong Hamiltonian} in the sense of \cite{Wang}. For each such a manifold, we prove that besides the ordinary two complex structures $J_\pm$ associated to the biHermitian description, there is a \emph{third} canonical complex structure $J_0$ underlying the geometry, which makes the manifold toric K$\ddot{a}$hler. We find the other generalized complex structure besides the symplectic one is always a B-transform of a generalized complex structure induced from a $J_0$-holomorphic Poisson structure $\beta$ characterized by an anti-symmetric constant matrix. Stimulated by the above results, we introduce a \emph{generalized Delzant construction} which starts from a Delzant polytope with $d$ faces of codimension 1, the standard K$\ddot{a}$hler structure of $\mathbb{C}^d$ and an anti-symmetric $d\times d$ matrix. This construction is used to produce non-abelian examples of strong Hamiltonian actions.
\end{abstract}
\section{Introduction}
Generalized K$\ddot{a}$hler (GK) structures are the generalized complex (GC) analogue of K$\ddot{a}$hler structures in complex geometry. It first appeared in the guise of biHermitian geometry in physicists' attempt to find the most general 2-dimensional $N=(2,2)$ supersymmetric $\sigma$-models. It was M. Gualtieri who first developed GK geometry in \cite{Gu00} \cite{Gu1} and found it is actually equivalent to the biHermitan geometry physicists are interested in.

Nontrivial GK manifolds are not that easy to construct. In the literature, the most successful method is probably provided by deforming a usual K$\ddot{a}$hler structure by a holomorphic Poisson structure. In this respect, R. Goto proved in \cite{Go1} a stability theorem of GK structures, which implies that deforming the complex structure of a compact K$\ddot{a}$hler manifold by a holomorphic Poisson structure can always be accompanied by a compatible deformation of the symplectic structure such that the resulting structure is a nontrivial GK structure. Other general formalisms for constructing GK manifolds include GK reduction of \cite{BCG1} \cite{LT}, Hamiltonian deformations of \cite{Gu1} and more recently GK blow-ups in \cite{Du}.

 Goto's result is in nature an existence theorem and beyond the search of examples. In \cite{Bou}, Boulanger started to investigate toric GK structures of symplectic type on a toric symplectic manifold $(M, \Omega, \mathbb{T}, \mu)$ ($\mathbb{T}$ is a torus acting on $M$ in a Hamiltonian fashion, $\mu$ the moment map and the symplectic form $\Omega$ provides one of the two underlying GC structures); in particular he found a subclass of such structures called \emph{anti-diagonal} ones, which could be effectively described by a strictly convex function $\tau$ defined in the interior of the moment polytope $\Delta$ and an anti-symmetric constant matrix $C$. This is an extension of the beautiful Abreu-Guillemin theory which characterizes a toric K$\ddot{a}$hler structure in terms of such a convex function referred to as a \emph{symplectic potential}. The anti-symmetric matrix $C$ in Boulanger's result seems occasionally to represent a holomorphic Poisson deformation as demonstrated in \cite[Corollary 2]{Bou} at least at the infinitesimal level. So in this respect, Boulanger's study can be viewed as an explicit realization of Goto's general theorem.

Anti-diagonal toric GK structures of symplectic type are also the central objects of the present paper. However, our approach towards these structures draws some essential new ideas from the author's recent work \cite{Wang}. In \cite{Wang}, to find an analogue of GIT (Geometric invariant theory) quotient in the GK setting, the author introduced the notion of \emph{strong Hamiltonian action} of a compact Lie group and investigated several elementary properties of this notion. Actually it is to look for more examples of strong Hamiltonian action that brought Boulanger's work into our sight. We shall revisit anti-diagonal toric GK structures of symplectic type from the viewpoint of \cite{Wang}, which facilitates the \emph{global characterization} of such structures, in contrast to Boulanger's mainly local results.

 The \emph{main result} we obtain in this paper are as follows. \emph{i}) For an anti-diagonal toric GK structure of symplectic type, its underlying GC structure $\mathbb{J}_1$ (not the one induced from $\Omega$) is always a B-field transform of a GC structure $\mathbb{J}_\beta$ induced from a holomorphic Poisson structure $\beta$ characterized by the anti-symmetric constant matrix $C$. \emph{ii}) Conversely, any toric K$\ddot{a}$hler structure together with an anti-symmetric constant matrix $C$ always gives rise to an anti-diagonal GK structure of symplectic type. A weaker version of part \emph{ii}) was already obtained by Boulanger from a viewpoint of deformation, \emph{requiring that the involved matrix $C$ be sufficiently small}. We don't need this restriction and thus in this respect our result goes beyond Goto's stability theorem which is basically a result of deformation theory.

 The key point of the above result \emph{i}) is to determine which complex structure we should refer to when we say $\beta$ is holomorphic. Underlying the GK structure are two complex structures $J_\pm$, but neither of them is the correct choice. A clue to the solution did appear in Boulanger's work: the symplectic potential $\tau$ in \cite{Bou} does determine a toric K$\ddot{a}$hler structure at least on the open dense subset $\mu^{-1}(\mathring{\Delta})\subset M$, where $\mathring{\Delta}$ is the interior of $\Delta$. Boulanger did prove that this K$\ddot{a}$hler structure is globally well-defined on $M$ in real dimension 4.

 Another clue comes from \cite{Wang}. For a compact connected Lie group $K$, the author proved in \cite{Wang} that a strong Hamiltonian $K$-action on a compact GK manifold can always be \emph{generalized complexified} (see \S~\ref{str} for details) and a stable orbit of this action acquires a natural $K$-equivariant K$\ddot{a}$hler structure from the ambient space; in particular its complex structure actually stems from the complexified group $K_\mathbb{C}$. We prove in this paper that, the condition of being anti-diagonal is equivalent to the requirement that the torus action be strong Hamiltonian. In this context, $\mu^{-1}(\mathring{\Delta})$ is consequently an orbit of the underlying generalized complex $\mathbb{T}_\mathbb{C}$-action and thus obtains a complex structure $J_0$ in this fashion. We prove that $J_0$ coincides with the complex structure underlying Boulanger's work. Since $\mathbb{T}_\mathbb{C}$ acts in a global way on $M$, we believe that $J_0$ should be globally defined on $M$ in general.

 To prove $J_0$ is actually global, we have refined Boulanger's compactification conditions by replacing a positive-definiteness condition with a non-degeneracy one, which is technically much easier for use. We also draw upon some important observations from \cite{Wang}. It should also be \emph{emphasized} that, though from a symplectic viewpoint the torus action on $M$ is fairly classical, we choose to fit it in the more general setting of extended actions of \cite{BCG1} and \cite{Wang}. Without this, some underlying geometric structures won't reveal themselves naturally.

 The detailed structure of the paper is outlined as follows. The first two sections \S~2, \S~3 are devoted to recalling the basics of GC geometry and (strong) Hamiltonian actions. The main body of the paper is \S~\ref{main} and \S~\ref{GDC}. After briefly recalling Abreu-Guillemin theory and Boulanger's generalization in \S~\ref{AGB}, in the local theory developed in \S~\ref{loc} we first slightly generalize Boulanger's observation (Thm.~\ref{GKG}). After that we prove for toric GK structures of symplectic type, being anti-diagonal is equivalent to that the torus action is strong Hamiltonian (Thm.~\ref{cha}). In Thm.~\ref{coi} and Thm.~\ref{hol}, we prove that the two complex structures on $\mu^{-1}(\mathring{\Delta})$ mentioned before do coincide and $\mathbb{J}_1$ is actually a B-transform of a GC structure induced from a holomorphic Poisson structure $\beta$. In \S~\ref{Com}, the problem of compactification is addressed. After having refined Boulanger's compactification condition in Lemma~\ref{c1} and Corollaries~\ref{c2}, \ref{c3}, we finally prove that the complex structure $J_0$ is actually globally defined on $M$ and any given toric K$\ddot{a}$hler structure together with an anti-symmetric constant matrix can produce a toric GK structure. \S~\ref{sub} is included for completeness to see what happens to $M\backslash\mu^{-1}(\mathring{\Delta})$. It is mainly an application of Goto's observations in \cite{Go} to our specified setting. \S~\ref{ex} contains an explicit example on $\mathbb{C}P^1\times \mathbb{C}P^1$ to demonstrate the general theory. According to our theorems, the starting point of constructing a toric GK structure is a toric K$\ddot{a}$hler structure and a \emph{canonical} choice is thus the one produced by the famous Delzant construction naturally associated to a Delzant polytope \cite{Del}. We then prove in \S~\ref{GDC} that toric GK structures constructed from this canonical K$\ddot{a}$hler structure can be interpreted from the viewpoint of GK reduction (Thm.~\ref{gde}), just as Delzant's construction is a successful application of K$\ddot{a}$hler reduction. This leads to the introduction of our \emph{generalized Delzant construction}. This procedure is used in \S~\ref{non} to provide examples of non-abelian strong Hamiltonian actions. The appendix collects some facts on matrices which are elementary but frequently (sometimes implicitly) used in the main text.
\section{Basics of generalized geometry}
Though we will finally deal with only compact toric symplectic manifolds, we do need some general background to lay the foundation of our approach to toric GK manifolds of symplectic type. In this section, we collect the necessary material of generalized geometry. The basic references are \cite{Gu00} and \cite{Gu0}.

A Courant algebroid $E$ is a real vector bundle $E$ over a smooth manifold $M$, together with an anchor map $\pi$ to $TM$, a non-degenerate inner product $(\cdot, \cdot)$ and a so-called Courant bracket $[\cdot , \cdot]_c$ on $\Gamma(E)$. These structures should satisfy some compatibility axioms we won't mention here. $E$ is called exact, if the short sequence \[0\longrightarrow T^*M\stackrel{\pi^*}\longrightarrow E \stackrel{\pi}\longrightarrow TM \longrightarrow0\]
is exact. In this paper, we only deal with exact ones. Given $E$, one can always find an isotropic right splitting $s:TM\rightarrow E$, with a curvature form $H\in \Omega_{cl}^3(M)$ defined by
\[H(X,Y,Z)=([s(X),s(Y)]_c,s(Z)),\quad X, Y, Z\in \Gamma(TM).\]
  By the bundle isomorphism $s+\pi^*:TM\oplus T^*M\rightarrow E$, consequently the Courant algebroid structure can be transported onto $TM\oplus T^*M$. Then the inner product $(\cdot,\cdot)$ is the natural pairing, i.e.
$( X+\xi,Y+\eta)=\xi(Y)+\eta(X)$, and the Courant bracket is
\begin{equation}[X+\xi, Y+\eta]_H=[X,Y]+\mathcal{L}_X\eta-\iota_Yd\xi+\iota_Y\iota_XH,\label{Courant}\end{equation}
called the $H$-twisted Courant bracket. Different splittings are related by B-tranforms: \[e^B(X+\xi)=X+\xi+B(X).\]
\begin{definition}
 A GC structure on a Courant algebroid $E$ is a complex structure $\mathbb{J}$ on $E$ orthogonal w.r.t. the inner product and its $\sqrt{-1}$-eigenbundle $L\subset E_\mathbb{C}$ is involutive under the Courant bracket. We also say $\mathbb{J}$ is integrable in this case.
\end{definition}
Since $\mathbb{J}$ and its $\sqrt{-1}$-eigenbundle $L$ are equivalent notions, we shall occasionally use them interchangeably to denote a GC structure. For $H\equiv0$, ordinary complex and symplectic structures are extreme examples of GC structures. Precisely, for a complex structure $I$ and a symplectic structure $\Omega$, the corresponding GC structures are of the following form:
\[\mathbb{J}_I=\left(
                 \begin{array}{cc}
                   -I & 0 \\
                   0 & I^* \\
                 \end{array}
               \right),\quad \mathbb{J}_\Omega=\left(
                                                 \begin{array}{cc}
                                                   0 & \Omega^{-1} \\
                                                   -\Omega & 0 \\
                                                 \end{array}
                                               \right).
\]
 A nontrivial example beyond these is produced by a holomorphic Poisson structure $\beta$: Let $\beta$ be a holomorphic Poisson structure relative to a complex structure $J$ on $M$. Then $L_\beta:=T_{0,1}M\oplus (\textup{Id}+\beta)T^*_{1,0}M$ is a GC structure $\mathbb{J}_\beta$. The matrix form of $\mathbb{J}_\beta$ is
\[\left(
      \begin{array}{cc}
        -J & -4\textup{Im}\beta \\
        0 & J^* \\
      \end{array}
    \right),
\]
where $\textup{Im}\beta$ is the imaginary part of $\beta$.
\begin{definition}A generalized metric on a Courant algebroid $E$ is an orthogonal, self-adjoint operator $\mathcal{G}$ such that $( \mathcal{G}\cdot,\cdot)$ is positive-definite on $E$. It is necessary that $\mathcal{G}^2=id$.
 \end{definition}
 A generalized metric induces a \emph{canonical} isotropic splitting: $E=\mathcal{G}(T^*M)\oplus T^*M$. It is called \emph{the metric splitting}. Given a generalized metric, we shall almost always choose its metric splitting to identify $E$ with $TM\oplus T^*M$. Then $\mathcal{G}$ is of the form $\left(\begin{array}{cc} 0 & g^{-1} \\g & 0 \\
\end{array} \right)$ where $g$ is an ordinary Riemannian metric.

  A generalized metric is an ingredient of a GK structure.
\begin{definition}
A GK structure on $E$ is a pair of commuting GC structures $(\mathbb{J}_1,\mathbb{J}_2)$ such that $\mathcal{G}=-\mathbb{J}_1 \mathbb{J}_2$ is a generalized metric.
\end{definition}
 A GK structure can also be reformulated in a biHermitian manner: There are two complex structures $J_\pm$ on $M$ compatible with the metric $g$ induced from the generalized metric. Let $\omega_\pm=gJ_\pm$. Then \emph{in the metric splitting} the GC structures and the corresponding biHermitian data are related by the Gualtieri map:
 \[\mathbb{J}_1=\frac{1}{2}\left(
  \begin{array}{cc}
    -J_+-J_-& \omega_+^{-1}-\omega_-^{-1} \\
    -\omega_++\omega_- & J_+^*+J_-^* \\
  \end{array}
\right),\quad \mathbb{J}_2=\frac{1}{2}\left(
  \begin{array}{cc}
    -J_++J_-& \omega_+^{-1}+\omega_-^{-1} \\
    -\omega_+-\omega_- & J_+^*-J_-^* \\
  \end{array}
\right).\]
Note that $\beta_1:=-\frac{1}{2}(J_+-J_-)g^{-1}$ and $\beta_2:=-\frac{1}{2}(J_++J_-)g^{-1}$ are actually real Poisson structures associated to $\mathbb{J}_{1}$ and $\mathbb{J}_{2}$ respectively.

If $\mathbb{J}_2$ is a B-transform of a GC structure $\mathbb{J}_\Omega$ induced from a symplectic form $\Omega$, the GK manifold $(M, \mathbb{J}_1, \mathbb{J}_2)$ is said to be \emph{of symplectic type}. It is known from \cite{En} that for a given symplectic form $\Omega$, compatible GC structures $\mathbb{J}_1$ which, together with a B-transform of $\mathbb{J}_\Omega$, form GK structures on $M$ are in one-to-one correspondence with \emph{tamed} integrable complex structures $J$ on $M$ such that its symplectic adjoint $J^{\Omega}:=-\Omega^{-1}J^*\Omega$ is also integrable. This fact greatly facilitates the study of GK structures of symplectic type. Precisely, in the metric splitting, if we set
\[\frac{1}{2}\left(
  \begin{array}{cc}
    -J_++J_-& \omega_+^{-1}+\omega_-^{-1} \\
    -\omega_+-\omega_- & J_+^*-J_-^* \\
  \end{array}
\right)=\left(\begin{array}{cc}
1 & 0\\
-b & 1\\\end{array}\right)\left(\begin{array}{cc}
0 & \Omega^{-1}\\
-\Omega & 0\\\end{array}\right)\left(\begin{array}{cc}
1 & 0\\
b & 1\\\end{array}\right),\]
then it can be simply obtained that
\begin{equation}J_-=J_+^\Omega=-\Omega^{-1} J_+^*\Omega,\quad g=-\frac{1}{2}\Omega (J_++J_-),\quad b=-\frac{1}{2}\Omega (J_+-J_-).\label{sy}\end{equation}
These are basic identities we will use frequently later. Recall that $J_+$ is tamed with $\Omega$ in the sense that the symmetric part of $-\Omega J_+$ is a Riemannian metric on $M$.
\section{Strong Hamiltonion actions and GK quotients}\label{str}
In this section, we recall the notion of (strong) Hamiltonian actions in the GK setting and some relevant results from \cite{LT} \cite{Wang}. The exposition here is slightly different from that of \cite{LT} in that we follow the spirit of \cite{BCG1} and stick to the metric splitting if a GK structure is in place.

Let $K$ be a connected Lie group (with Lie algebra $\mathfrak{k}$) acting on the manifold $M$ from the left. Then the infinitesimal action $\varphi_0:\mathfrak{k}\rightarrow \Gamma(TM)$ is a Lie algebra homomorphism. Since in generalized geometry $TM$ is replaced by a Courant algebroid $E$, we would like to lift the $\mathfrak{k}$-action to $E$.
\begin{definition}A map $\varphi: \mathfrak{k}\rightarrow \Gamma(E)$ covering $\varphi_0$ is called an isotropic trivially extended $\mathfrak{k}$-action if i) $\varphi$ is isotropic, i.e. the image of $\varphi$ is isotropic pointwise w.r.t. the inner product and ii) $\varphi$ preserves the brackets, i.e. $\varphi([\varsigma,\zeta])=[\varphi(\varsigma), \varphi(\zeta)]_c$ for $\varsigma, \zeta\in \mathfrak{k}$.
\end{definition}

If this action integrates to a $K$-action on $E$, we call it an isotropic trivially extended $K$-action on $M$. There is a fairly general theory of extended $K$-actions in \cite{BCG1}, but we don't need such generality. \emph{In the rest of the paper, when referring to an extended $\mathfrak{k}$-action (or $K$-action), we always mean an isotropic trivially extended one.}

Given a splitting of $E$ preserved by the extended action, the infinitesimal action can be written in the form $\varphi(\varsigma)=X_\varsigma+\xi_\varsigma$, where $X_\varsigma$ is the vector field generated by $\varsigma$ and $\xi_{(,)}:\mathfrak{k}\rightarrow \Omega^1(M)$ is a $\mathfrak{k}$-equivariant map such that
\begin{equation}\xi_\varsigma(X_\zeta)+\xi_\zeta(X_\varsigma)=0,\quad d\xi_\varsigma=\iota_{X_\varsigma}H.\end{equation}

If the underlying $K$-action on $M$ of an extended $K$-action is proper and free, the Courant algebroid $E$ descends to the quotient $M/K$. Precisely, let $\mathbb{K}$ be the subbundle of $E$ generated by the image of $\varphi$ and $\mathbb{K}^\bot\subset E$ the orthogonal of $\mathbb{K}$ w.r.t. the inner product. Then $\mathbb{K}\subset \mathbb{K}^\bot$ and we can obtain a Courant algebroid $E_{red}:=\frac{\mathbb{K}^\bot}{\mathbb{K}}/K$ whose Courant bracket and inner product can be derived from the Courant bracket of $K$-invariant sections of $E$. If $\mathcal{D}\subset E_\mathbb{C}$ is a Dirac structure (involutive maximal isotropic subbundle), then $\mathcal{D}$ descends to a Dirac structure on the quotient under good conditions, e.g. $\mathcal{D}\cap \mathbb{K}_\mathbb{C}$ has constant rank. The reduced version of $\mathcal{D}$ is $\frac{\mathcal{D}\cap \mathbb{K}_\mathbb{C}^\bot+\mathbb{K}_\mathbb{C}}{\mathbb{K}_\mathbb{C}}/K$.

Let a \emph{compact} connected Lie group $K$ act on a GK manifold $(M, \mathbb{J}_1, \mathbb{J}_2)$ in the extended manner, preserving the GK structure on $M$ and consequently the metric splitting. The notion of (generalized) moment map can be defined in the more general context of GC manifolds. As a convention, when referring to a moment map in the GK setting, we always mean one associated to $\mathbb{J}_2$.
\begin{definition}(\cite{LT} \cite{Wang}) Let $M$ be a GK manifold carrying an extended $K$-action $\varphi$ preserving the underlying GK structure. An equivariant map $\mu: M\rightarrow \mathfrak{k}^*$ is called a moment map, if
\begin{equation}\mathbb{J}_2(X_\varsigma+\xi_\varsigma-\sqrt{-1}d\mu_\varsigma)=\sqrt{-1}(X_\varsigma+\xi_\varsigma-\sqrt{-1}d\mu_\varsigma)\label{gmm}\end{equation}
for any $\varsigma\in \mathfrak{k}$. If this happens, the action is called Hamiltonian. If additionally $\beta_1(d\mu_\varsigma, d\mu_\zeta)=0$ for any $\varsigma, \zeta\in \mathfrak{k}$, we say the action is strong Hamiltonian. Note that $\beta_1$ is the Poisson structure associated to $\mathbb{J}_1$.
\end{definition}
A more friendly reformulation of the strong Hamiltonian condition $\beta_1(d\mu_\varsigma, d\mu_\zeta)=0$ is $(\mathbb{J}_1(X_\varsigma+\xi_\varsigma), X_\zeta+\xi_\zeta)=0$, which is motivated by the attempt to complexify the extended $\mathfrak{k}$-action at the infinitesimal level \cite{Wang}.

\begin{example}Let $(M, \mathbb{J}_1, \mathbb{J}_2)$ be a GK manifold of symplectic type where $\mathbb{J}_2$ is a B-transform of $\mathbb{J}_\Omega$ induced from a symplectic form $\Omega$. If $(M, \Omega)$ is equipped with an ordinary Hamiltonian $K$-action which also preserves the underlying complex structure $J_+$ (consequently the whole GK structure), then \emph{in the metric splitting}, this action involves a cotangent part and becomes an extended one. Precisely, at the infinitesimal level, for $\varsigma\in \mathfrak{k}$ we have $\varphi(\varsigma)=X_\varsigma-b(X_\varsigma)$,
where $b$ is the 2-form $b=-1/2\Omega(J_+-J_-)$. This extended action is obviously a Hamiltonian one. This example is actually the only one we are finally concerned with.
\end{example}

All Hamiltonian K$\ddot{a}$hler manifolds are strong Hamiltonian for the relevant Poisson structure $\beta_1$ is zero. Due to dimensional reason, the first nontrivial examples of strong Hamiltonian actions are provided by Hamiltonian $S^1$-actions on GK manifolds. One of our motivations in this paper is to present some more complicated examples.

The importance of the notion of strong Hamiltonian actions lies in the fact that the group actions can be \emph{generalized complexified} in a fashion compatible with the underlying geometry and one can develop the GIT aspect of GK reduction \cite{Wang}. A sketch of this theory goes as follows:

Eq.~(\ref{gmm}) is equivalent to the following equation
\begin{equation}
J_+X_\varsigma^+=J_-X_\varsigma^-=-g^{-1}d\mu_\varsigma,\label{cent}
\end{equation}
where $X_\varsigma^\pm=X_\varsigma\pm g^{-1}\xi_\varsigma$. Let $Y_\varsigma=J_+X_\varsigma^+$. It's easy to show $\mathbb{J}_1(X_\varsigma+\xi_\varsigma)=-Y_\varsigma$ and we can define the action of $\mathfrak{k}_\mathbb{C}=\mathfrak{k}+\sqrt{-1}\mathfrak{k}$ in a manner similar to the ordinary case of complexifying a holomorphic action of a compact group on a complex manifold:
\[\varphi_\mathbb{C}(\varsigma+\sqrt{-1}\zeta)=X_\varsigma+\xi_\varsigma-\mathbb{J}_1(X_\zeta+\xi_\zeta)=X_\varsigma+\xi_\varsigma+Y_\zeta.\]
That the extended $K$-action is strong Hamiltonian assures that the above map is really a Lie algebra homomorphism from $\mathfrak{k}_\mathbb{C}$ to $E$. Under suitable conditions (e.g. $M$ is compact), this Lie algebra action can integrate to an extended $K_\mathbb{C}$-action on $M$. We shall call this action a GC action to emphasize its origin. The GC action will preserve $\mathbb{J}_1$ and its associated Poisson structure $\beta_1$.

For simplicity, let $M$ be compact and $K$ act on $\mu^{-1}(0)$ freely, then as Lin-Tolman showed in \cite{LT}, the quotient $\mu^{-1}(0)/K$ acquires a natural GK structure in a manner similar to the classical K$\ddot{a}$hler reduction. On the other hand, $\mathcal{Z}:=K_\mathbb{C}\mu^{-1}(0)$ is an open submanifold of $M$ and inherits a GC structure $\mathbb{J}_1|_{\mathcal{Z}}$. A surprising result in \cite{Wang} is that each $K_\mathbb{C}$-orbit $\mathcal{O}\subset\mathcal{Z}$ acquires a natural Hamiltonian $K$-equivariant K$\ddot{a}$hler structure from the ambient space in the following fashion: on $\mathcal{O}$ one can simply define a complex structure $J_0$ by letting $J_0 X_\varsigma=Y_\varsigma$ and define a metric $\tilde{g}$ by letting
\begin{equation}\tilde{g}(X_\varsigma, X_\zeta)=\tilde{g}(Y_\varsigma, Y_\zeta)=g(X_\varsigma^+, X_\zeta^+),\quad \tilde{g}(X_\varsigma, Y_\zeta)=g(X_\varsigma, Y_\zeta).\label{kah}\end{equation}
$J_0$ is actually the complex structure induced from the standard one on $K_\mathbb{C}$ and $\tilde{g}$ is the restriction of the generalized metric $\mathcal{G}$ on the bundle $\mathbb{K}\oplus \mathbb{J}_1\mathbb{K}$. The quotient of $(\mathcal{Z}, \mathbb{J}_1|_{\mathcal{Z}})$ by the GC action of $K_\mathbb{C}$ is a GC manifold and can be viewed as the GIT quotient in the GK setting. In \cite{Wang}, the author has proved that this quotient can be identified with the GK quotient in the sense of \cite{LT} and the reduced version of $\mathbb{J}_1|_\mathcal{Z}$ is precisely the one obtained from Lin-Tolman's story. This is a natural extension of the well-known GIT aspect of K$\ddot{a}$hler reduction saying that the GIT quotient coincides with the symplectic quotient \cite{Kir}.
\section{Anti-diagonal toric GK structures of symplectic type}\label{main}
\subsection{Abreu-Guillemin theory and Boulanger's generalization}\label{AGB}
Let us recall briefly the Abreu-Guillemin theory first. A recent fairly readable account on this topic can be found in \cite{Ap}.
\begin{definition}A toric symplectic manifold $(M, \Omega, \mathbb{T}, \mu)$ of dimension $2n$ is a symplectic manifold $(M, \Omega)$ with an effective and Hamiltonian action of the $n$-dimensional torus $\mathbb{T}=\mathbb{T}^n$. Note that here $\mu$ is the moment map.
\end{definition}
Let $(M,\Omega, \mathbb{T}, \mu)$ be a compact toric symplectic manifold and $\mathfrak{t}\cong \mathbb{R}^n$ the Lie algebra of $\mathbb{T}$. By the famous convexity theorem of Atiyah-Guillemin-Sternberg \cite{At} \cite{GS}, the image $\Delta$ of $\mu$ is a polytope in $\mathfrak{t}^*=(\mathbb{R}^n)^*$ which is the convex hull of the image of fixed points of the torus action. $\Delta$ is thus called the \emph{moment polytope}. A famous theorem of Delzant states that compact toric symplectic manifolds are classified by their moment polytopes $\Delta$ up to equivariant symplectomorphism \cite{Del}.

Given a compact toric symplectic manifold $(M,\Omega, \mathbb{T}, \mu)$, Guillemin in \cite{Gu1} showed that compatible $\mathbb{T}$-invariant K$\ddot{a}$hler structures are actually also determined by data specified on the moment polytope $\Delta$. The theory is partially sketched below.

Let $\mathring{\Delta}$ be the interior of $\Delta$. Then $\mathring{M}:=\mu^{-1}(\mathring{\Delta})$ consists of points at which $\mathbb{T}$ acts freely. Topologically, $\mathring{M}=\mathring{\Delta}\times \mathbb{T}$, i.e. a trivial principal $\mathbb{T}$-bundle $\mu: \mathring{M}\rightarrow\mathring{\Delta}$. Denote the set of $\mathbb{T}$-invariant complex structures on $M$ compatible with $\Omega$ by $K_\Omega^{\mathbb{T}}(M)$, i.e. the set of toric K$\ddot{a}$hler structures on $M$. Let $I\in K_\Omega^{\mathbb{T}}(M)$ and $\{X_j\}$ be the fundamental vector fields corresponding to a basis $\{e_j\}$ of $\mathfrak{t}$. Then $\{X_j, IX_j\}$ forms a global frame of $T\mathring{M}$ and the Lie bracket of any two vector fields in this frame vanishes. Let $\{\zeta_j,\vartheta_j\}$ be the dual frame on $T^*\mathring{M}$. Then $d\zeta_j=d\vartheta_j=0$ and hence locally $\zeta_j=d\theta_j$ and $\vartheta_j=du_j$. $\{\theta_j+\sqrt{-1}u_j\}$ is thus a local holomorphic coordinate system on $\mathring{M}$ (actually, these $u_j$'s are globally defined on $\mathring{M}$ due to the fact that $\mathring{\Delta}$ is simply connected). On the other side, $\{\vartheta_j\}$ and $\{d\mu_j\}$ determine the same integrable Lagrangian distribution $\mathcal{D}$ generated by $\{X_j\}$ and thus these $u_j$'s are all functions depending only on $\mu$, i.e.
\begin{equation}du_j=-\sum_{k=1}^n\phi_{kj}(\mu)d\mu_k,\label{comp}\end{equation}
or \footnote{As a convention, we have written $d\theta_j$'s or $d\mu_j$'s in a column. Similar notation is used below.}
\[I^*\left(
       \begin{array}{c}
         d\theta \\
         d\mu \\
       \end{array}
     \right)=\left(
               \begin{array}{cc}
                 0 & \phi \\
                 -\phi^{-1} & 0 \\
               \end{array}
             \right)\left(
                      \begin{array}{c}
                        d\theta \\
                        d\mu \\
                      \end{array}
                    \right).\]
It is easy to check that $\theta_j, \mu_j$ are Darboux coordinates, i.e. on $\mathring{M}$
\[\Omega=\sum_{j=1}^nd\mu_j\wedge d\theta_j.\]
Compatibility of $I$ with $\Omega$ forces the matrix $\phi=(\phi_{kj})$ to be symmetric and positive-definite, and integrability of Eq.~(\ref{comp}) implies that $\phi$ must be the Hessian of a function $\tau$ defined on $\mathring{\Delta}$ (positive-definiteness of $\phi$ then means that $\tau$ is strictly convex). Due to the cental role of $\tau$, it is called the \emph{symplectic potential} of the invariant K$\ddot{a}$hler structure $I$, \footnote{By abuse of language, we call $I$ a K$\ddot{a}$hler structure for $I$ completely determines the K$\ddot{a}$hler structure in our present setting.} which provides a very practical computational tool in examining geometric ideas. Some distinguished applications can be found in \cite{Ab}.

Boulanger's generalization went in a similar spirit. He considered $\mathbb{T}$-invariant GK structures $(\mathbb{J}_1, \mathbb{J}_2)$ of symplectic type on $(M,\Omega, \mathbb{T}, \mu)$, where  $\mathbb{J}_2$ is a B-transform of $\mathbb{J}_\Omega$. The above $I$ is replaced by $J_+$ underlying the biHermitian description. However, the relaxed condition of tameness no longer ensures that $\theta_j, \mu_j$ be Darboux coordinates in general. Boulanger then introduced a more restrictive subclass to reserve this property. Denote the space of $\mathbb{T}$-invariant GK structures of symplectic type by $GK_\Omega^{\mathbb{T}}(M)$. Then an element of $GK_\Omega^{\mathbb{T}}(M)$ is called \emph{anti-diagonal} if for the underlying complex structures $J_\pm$ the following condition holds:
\begin{equation}J_+\mathcal{D}=J_-\mathcal{D},\end{equation}
where $\mathcal{D}$ is again the Lagrangian distribution generated by $\{X_j\}$.

Before proceeding further, let us introduce some notation. Following Boulanger, we denote this subset of anti-diagonal elements in $GK_\Omega^{\mathbb{T}}(M)$ by $DGK_\Omega^{\mathbb{T}}(M)$. Since an element in $GK_\Omega^{\mathbb{T}}(M)$ is completely parameterized by its associated complex structure $J_+$, we usually write $J_+\in GK_\Omega^{\mathbb{T}}(M)$ to imply this fact. Sometimes we also write $\mathbb{J}_1\in GK_\Omega^{\mathbb{T}}(M)$ if we want to emphasize the GC aspect of the underlying structures.

For a $J_+\in DGK_\Omega^{\mathbb{T}}(M)$, $\theta_j, \mu_j$ are again Darboux coordinates (Boulanger called them \emph{admissible coordinates associated to $J_+$}) and with such coordinates $J_\pm$ are of a form similar to Abreu-Guillemin's case
\begin{equation}J_+^*\left(
       \begin{array}{c}
         d\theta \\
         d\mu \\
       \end{array}
     \right)=\left(
               \begin{array}{cc}
                 0 & \phi^T \\
                 -(\phi^{-1})^T & 0 \\
               \end{array}
             \right)\left(
                      \begin{array}{c}
                        d\theta \\
                        d\mu \\
                      \end{array}
                    \right),\quad J_-^*\left(
       \begin{array}{c}
         d\theta \\
         d\mu \\
       \end{array}
     \right)=\left(
               \begin{array}{cc}
                 0 & \phi \\
                 -\phi^{-1} & 0 \\
               \end{array}
             \right)\left(
                      \begin{array}{c}
                        d\theta \\
                        d\mu \\
                      \end{array}
                    \right)
\label{J+-}\end{equation}
except that $\phi$ is not necessarily symmetric. Note that here $\phi^T$ denotes the transpose of $\phi$. Integrability of $J_\pm$ then forces the symmetric part $\phi_s$ ($=(\phi+\phi^T)/2$) of $\phi$ to be the Hessian of a function $\tau$ on $\mathring{\Delta}$ and the anti-symmetric part $\phi_a$ ($=(\phi-\phi^T)/2$) to be a constant anti-symmetric $n\times n$ matrix. Tameness then is simply equivalent to that $\tau$ is strictly convex. A detailed argument can be found in the next subsection in a more general setting. \emph{Another fact we should emphasize is that $J_+\in GK_\Omega^{\mathbb{T}}(M)$ lies in $DGK_\Omega^{\mathbb{T}}(M)$ if and only if for certain admissible coordinates, the matrix form of $J_+$ is anti-diagonal.}

If the matrix form of $J_+\in DGK_\Omega^{\mathbb{T}}(M)$ is anti-diagonal w.r.t. the admissible coordinates $\theta_j, \mu_j$ associated to a reference element $I_+\in DGK_\Omega^{\mathbb{T}}(M)$, for later use and also for the reader's convenience, we collect below the matrix forms of the underlying several geometric structures viewed as linear maps:
\[J_+\sim
\left(
  \begin{array}{cc}
    0 & -\phi^{-1} \\
    \phi  & 0 \\
  \end{array}
\right),\quad J_-\sim
\left(
  \begin{array}{cc}
    0 & -(\phi^{-1})^T \\
    \phi^T  & 0 \\
  \end{array}
\right),\quad g\sim\left(
               \begin{array}{cc}
                 (\phi^{-1})_s & 0 \\
                 0 & \phi_s \\
               \end{array}
             \right),
\]
\[b\sim\left(
               \begin{array}{cc}
                 (\phi^{-1})_a & 0 \\
                 0 & \phi_a \\
               \end{array}
             \right),\quad
     \Omega\sim\left(
               \begin{array}{cc}
                 0 & -\textup{I} \\
                 \textup{I} & 0 \\
               \end{array}
             \right),\quad \beta_1\sim\left(
                              \begin{array}{cc}
                                0 & -\phi_a(\phi_s)^{-1} \\
                                -(\phi_s)^{-1}\phi_a & 0 \\
                              \end{array}
                            \right).
                    \]
These are computed relative to $\{d\theta_j, d\mu_j\}$ and $\{\partial_{\theta_j}, \partial_{\mu_j}\}$, and "$\textup{I}$" denotes the $n\times n$ identity matrix. It should be pointed out that due to our notation, the matrix form of the composition $RS$ of two maps $R$ and $S$ is $\hat{S}\hat{R}$ rather than $\hat{R}\hat{S}$, if $\hat{R}$, $\hat{S}$ are the matrix forms of $R, S$ respectively. Another fact we shall mention here is that by abuse of language, we will not distinguish $\mathbb{T}$-invariant smooth functions on $M$ (or $\mathring{M}$) from smooth functions on $\Delta$ (or $\mathring{\Delta}$) as is often done in the literature.

\subsection{Local theory}\label{loc}
Given a compact toric symplectic manifold $(M,\Omega, \mathbb{T}, \mu)$, in the following three subsections, we set it our goal to characterize elements in $DGK_\Omega^{\mathbb{T}}(M)$ step by step. To familiarize the reader with the technical background, let us begin with a slight generalization of Boulanger's observation.

 Fix a basis $\{e_j\}$ of $\mathfrak{t}$ and let $\{\mu_j\}$ be the corresponding components of $\mu$. Note again that $\mathring{M}$ is a trivial principal $\mathbb{T}$-bundle over $\mathring{\Delta}$. Let $\zeta=\sum_j\zeta_je_j$ be a flat connection of this $\mathbb{T}$-bundle. Due to the fact that the vertical distribution is Lagrangian, we must have a 1-form $\sigma_\zeta=\sum_j h_jd\mu_j$ with $h_j$ depending only on $\mu$ such that
\[\Omega=\sum_j d\mu_j\wedge \zeta_j+d\sigma_\zeta.\]
We call the matrix $F_\zeta:=(h_{k,j}-h_{j,k})$ the associated matrix of the connection $\zeta$. Obviously, $F_\zeta$ is determined by $\zeta$. If $F_\zeta$ happens to be a constant matrix, we say $\zeta$ is an \emph{admissible connection}.

\begin{theorem}\label{GKG}An element $J_+\in GK_\Omega^{\mathbb{T}}(\mathring{M})$ is determined by a triple $(\zeta,\tau, C)$ where $\zeta$ is an admissible connection on $\mathring{M}$, $C$ is an anti-symmetric $n\times n$ constant matrix and $\tau$ is a strictly convex function on $\mathring{\Delta}$ such that its Hessian $\phi_s$ makes the matrix \[\phi_s+\frac{1}{4}F_\zeta(\phi_s)^{-1}F_\zeta\] positive-definite. Conversely, such a triple $(\zeta,\tau, C)$ also gives rise to an element in $GK_\Omega^{\mathbb{T}}(\mathring{M})$.\label{GK}
\end{theorem}
\begin{proof}Let $J_+\in GK_\Omega^{\mathbb{T}}(\mathring{M})$ and $X_j$ be the fundamental vector field generated by $e_j$. Tameness of $J_+$ with $\Omega$ assures that $\{X_j, J_+X_j\}$ be a global frame of $T\mathring{M}$. Let $\{\zeta_j, \vartheta_j\}$ be the corresponding dual frame of $T^*\mathring{M}$. Since $J_+$ is integrable and the action of $\mathbb{T}$ is abelian, $\zeta:=\sum_j\zeta_je_j$ gives rise to a flat connection on $\mathring{M}$. Locally $\zeta_j=d\theta_j$, $\vartheta_j=du_j$ and $\{\theta_j+\sqrt{-1}u_j\}$ is a local $J_+$-holomorphic coordinate system on $\mathring{M}$ ($u_j$'s are actually global on $\mathring{M}$). Then just as outlined in the previous subsection $du_j=-\sum_k\phi_{kj}d\mu_k$,
where $\phi_{kj}$'s are functions only of $\mu$, and
\begin{equation}J_+^*\left(
       \begin{array}{c}
         d\theta \\
         d\mu \\
       \end{array}
     \right)=\left(
               \begin{array}{cc}
                 0 & \phi^T \\
                 -(\phi^{-1})^T & 0 \\
               \end{array}
             \right)\left(
                      \begin{array}{c}
                        d\theta \\
                        d\mu \\
                      \end{array}
                    \right).\label{J+}\end{equation}
We now know that as mentioned before for some functions $h_j$
\[\Omega=\sum_j d\mu_j\wedge d\theta_j+d(\sum_jh_jd\mu_j),\]
or equivalently with respect to $d\theta, d\mu$
\[\Omega\sim \left(
               \begin{array}{cc}
                 0 & -\textup{I} \\
                 \textup{I} & F \\
               \end{array}
             \right), \quad \textup{where}\quad F=F_\zeta=(h_{k,j}-h_{j,k}).
\]

The same argument also applies to $J_-$. There should be a local coordinate system $\{\theta_j', \mu_j\}$ and a matrix-valued function $\psi$ only of $\mu$ such that
\[J_-^*\left(
       \begin{array}{c}
         d\theta' \\
         d\mu \\
       \end{array}
     \right)=\left(
               \begin{array}{cc}
                 0 & \psi^T \\
                 -(\psi^{-1})^T & 0 \\
               \end{array}
             \right)\left(
                      \begin{array}{c}
                        d\theta' \\
                        d\mu \\
                      \end{array}
                    \right).\]
We shall prove that $\psi$ is nothing else but $\phi^T$. Actually,
\begin{eqnarray*}d\theta'&=&-\psi^TJ_-^*d\mu=\psi^T\Omega J_+\Omega^{-1}(d\mu)\\
&=&-\psi^T\Omega J_+(\partial_\theta)=\psi^T\phi^{-1}\Omega(\partial_\mu)\\
&=&\psi^T\phi^{-1}(d\theta+F d\mu).\end{eqnarray*}
However, since $d\theta$ and $d\theta'$ are both flat connections on the same principal $\mathbb{T}$-bundle, we must have
\[d\theta'_j=d\theta_j+df_j\]
for some functions $f_j$ depending only on $\mu$. This observation implies $\psi^T\phi^{-1}=\textup{I}$ or equivalently $\psi=\phi^T$ as required. Additionally, we must also have $F_{kj}=f_{j,k}$. Therefore,
\[F_{kj,l}=f_{j,kl}=f_{j,lk}=F_{lj,k},\]
and furthermore,
\[F_{kj,l}=F_{lj,k}=-F_{jl,k}=-F_{kl,j}=F_{lk,j}=F_{jk,l}=-F_{kj,l}.\]
Consequently $F_{kj,l}=0$ and $F$ is actually an anti-symmetric constant matrix, i.e. $\zeta$ is an admissible connection.

Since $d\theta_j-\sqrt{-1}\sum_k\phi_{kj}d\mu_k$ and $d\theta_j'-\sqrt{-1}\sum_k\phi_{jk}d\mu_k$ are holomorphic 1-forms w.r.t. $J_+$ and $J_-$ respectively, integrability of $J_\pm$ thus implies
\begin{equation}\phi_{kj,l}=\phi_{lj,k}, \quad \phi_{jk,l}=\phi_{jl,k}.\label{pr}\end{equation}
Then we can conclude just as Boulanger had done in \cite{Bou} that the anti-symmetric part $\phi_a$ of $\phi$ should be a constant matrix $C$ and the symmetric part $\phi_s$ of $\phi$ be the Hessian of a function $\tau$ defined on $\mathring{\Delta}$. For the reader's convenience, we sketch the argument below. From Eq.~(\ref{pr}), we find
\[\phi_{kj,l}=\phi_{lj,k}=\phi_{lk,j}=\phi_{jk,l},\]
inducing that $(\phi_a)_{kj,l}=0$ and thus that $\phi_a$ should be an anti-symmetric constant matrix. With this fact in mind, Eq.~(\ref{pr}) is simplified:
\[(\phi_s)_{kj,l}=(\phi_s)_{lj,k}=(\phi_s)_{lk,j},\]
which implies that $\phi_s$ should be the Hessian of a function $\tau$ defined on $\mathring{\Delta}$. Note that the simply-connectedness of $\mathring{\Delta}$ should be used here to obtain the global existence on $\mathring{\Delta}$ of $\tau$.

To see what tameness of $J_+$ with $\Omega$ means, we should derive the matrix form of the metric $g$. Note that w.r.t. the frame $\{d\theta, d\mu\}$,
 \[J_-^*\sim \left(
               \begin{array}{cc}
                 \textup{I} & -F \\
                 0 & \textup{I} \\
               \end{array}
             \right)\left(
                      \begin{array}{cc}
                        0 & \phi \\
                        -\phi^{-1} & 0 \\
                      \end{array}
                    \right)\left(
               \begin{array}{cc}
                 \textup{I} & F \\
                 0 & \textup{I} \\
               \end{array}
             \right)=\left(
                       \begin{array}{cc}
                         F\phi^{-1} & F\phi^{-1}F+\phi \\
                         -\phi^{-1} & -\phi^{-1}F \\
                       \end{array}
                     \right).
 \]
 Then due to the formula $g=1/2(J_+^*+J_-^*)\Omega$, we can find the matrix form of $g$ relative to $\{d\theta, d\mu\}$:
 \[g\sim \left(
           \begin{array}{cc}
             (\phi^{-1})_s & \phi^{-1}F/2 \\
             -F(\phi^{T})^{-1}/2 & \phi_s\\
           \end{array}
         \right).
 \]
 It's elementary to find that positive-definiteness of $g$ is equivalent to that both $\phi_s$ and $\phi_s+1/4F(\phi_s)^{-1}F$ are positive-definite. This shows $\tau$ should satisfy the properties listed in the theorem. Clearly, the triple $(\zeta,\tau, C)$ determines $J_+$ completely.

 Conversely, given the triple $(\zeta,\tau, C)$ satisfying the conditions listed in the theorem, let $\phi_s$ be the Hessian of $\tau$ and $\phi=\phi_s+C$ and define $J_+$ in the manner of Eq.~(\ref{J+}). Obviously such a $J_+\in GK_\Omega^{\mathbb{T}}(\mathring{M})$.
\end{proof}
\emph{Remark}. Several comments are disirable here. i) In our more general case, the matrix forms of $J_+^*$ and $J_-^*$ are not necessarily anti-diagonal in the fixed frame $\{d\theta, d\mu\}$. This happens only if $F=0$, which is precisely what Boulanger called an anti-diagonal toric GK structure of symplectic type. ii) $\{\theta_j, \mu_j\}$ has to be a Darboux coordinate system in Boulanger's case while in general this is not the case; iii) Our theorem reveals that the broader case of general toric GK structures of symplectic type is also, to some extent, very well-behaved. We will analyze such structures in a forthcoming paper.

The following theorem is our first basic observation in this paper, motivated by the general theory developed by the author in \cite{Wang}.
\begin{theorem}\label{cha}
Let $\mathbb{J}_1\in GK_\Omega^\mathbb{T}(M)$. Then $\mathbb{J}_1\in DGK_\Omega^\mathbb{T}(M)$ if and only if the extended $\mathbb{T}$-action on $M$ is strong Hamiltonian.
\end{theorem}
\begin{proof}Note that the notion of strong Hamiltonian action is independent of the splitting we choose. We are safe here if we only use the notion of Hamiltonian action in the ordinary symplectic sense. Obviously we only need to prove the theorem on $\mathring{M}$. Let $\{e_j\}$ be a basis of $\mathfrak{t}$, $\{X_j\}$ its associated fundamental vector fields and $\{\mu_j\}$ the corresponding components of $\mu$. First note that due to Eq.~(\ref{sy}),
\[\beta_1=-\frac{1}{2}(J_+-J_-)g^{-1}=(J_+-J_-)(J_++J_-)^{-1}\Omega^{-1}.\]
Therefore,
\[\beta_1(d\mu_j)=(J_+-J_-)(J_++J_-)^{-1}\Omega^{-1}(d\mu_j)=-(J_+-J_-)(J_++J_-)^{-1}X_j,\]
and we can claim that the extended $\mathbb{T}$-action is strong Hamiltonian if and only if
\[(J_+-J_-)(J_++J_-)^{-1}\mathcal{D}\subset \mathcal{D}\]
where $\mathcal{D}$ is the Lagrangian distribution generated by $\{X_j\}$. Note that
\[(J_++J_-)(J_+-J_-)=-(J_+-J_-)(J_++J_-)\]
and thus
\begin{eqnarray*}
&\quad&(J_+-J_-)(J_++J_-)^{-1}\mathcal{D}\subset \mathcal{D}\Leftrightarrow (J_++J_-)^{-1}(J_+-J_-)\mathcal{D}\subset \mathcal{D}\\
&\Leftrightarrow &[\textup{Id}-2(J_++J_-)^{-1}J_-]\mathcal{D}\subset \mathcal{D}
\Leftrightarrow (J_++J_-)^{-1}J_-\mathcal{D}\subset \mathcal{D}\\
&\Leftrightarrow& (J_++J_-)^{-1}J_-\mathcal{D}=\mathcal{D}
\Leftrightarrow J_-(J_++J_-)\mathcal{D}=\mathcal{D}\\
&\Leftrightarrow& (J_-J_+-\textup{Id})\mathcal{D}=\mathcal{D}
\Leftrightarrow J_-J_+\mathcal{D}=\mathcal{D}\Leftrightarrow J_+\mathcal{D}=J_-\mathcal{D}.
\end{eqnarray*}
The proof is thus completed.
\end{proof}

From now on, we will consider $J_+\in DGK_\Omega^\mathbb{T}(M)$ only unless otherwise stated. To motivate our further steps, some general discussion concerning the implication of Thm.~\ref{cha} is desirable here. Due to the observation in \cite{Wang}, in the metric splitting, in terms of the admissible coordinates $\{\theta_j, \mu_j\}$ associated with $J_+$, the infinitesimal extended action of $\mathfrak{t}$ on $\mathring{M}$ is given by
\begin{equation}X_j+\xi_j:=\partial_{\theta_j}-b(\partial_{\theta_j})=\partial_{\theta_i}-\sum_k(\phi^{-1})_a^{kj}d\theta_k,\quad j=1,2,\cdots, n.\label{ga}\end{equation}
 Now that the extended $\mathbb{T}$-action on $M$ is strong Hamiltonian, it can be generalized complexified. Thus we have a GC action of $\mathbb{T}_\mathbb{C}$ on $M$. Especially in the present setting,  $\mathcal{Z}:=\mathbb{T}_\mathbb{C}\mu^{-1}(0)=\mathring{M}$ (we can make a translation if necessary to assure that $0$ be a regular value of $\mu$) is itself a $\mathbb{T}_\mathbb{C}$-orbit. Therefore, from the general theory of \cite{Wang}, $\mathring{M}$ has to be a $\mathbb{T}$-invariant K$\ddot{a}$hler manifold whose complex structure $J_0$ is actually induced from $\mathbb{T}_\mathbb{C}$. Furthermore, as $\mathbb{T}_\mathbb{C}$ acts on $M$ in a global fashion, it is hopeful that this $J_0$ would be actually a global complex structure on $M$.

On the other side, there does exist a compatible K$\ddot{a}$hler structure on $\mathring{M}$ from Boulanger's theory: Since the symplectic potential $\tau$ underlying Boulanger's result is a strictly convex function on $\mathring{\Delta}$, due to Abreu-Guillemin theory, it should define a K$\ddot{a}$hler structure at least on $\mathring{M}$, whose complex structure $J_0'$ has the matrix form below
\begin{equation}J_0'\sim \left(
            \begin{array}{cc}
              0 & -(\phi_s)^{-1} \\
              \phi_s & 0 \\
            \end{array}
          \right),\label{avc}
\end{equation}
where $\phi_s$ is the Hessian of $\tau$. We shall reasonably call $J_0'$ the average of $J_+$ and $J_-$ in the sense that $\phi_s=(\phi+\phi^T)/2$.

The following theorem shows the two complex structures $J_0$ and $J_0'$ on $\mathring{M}$ are actually the same. This provides an intrinsic explanation of the origin of Boulanger's symplectic potential $\tau$--it is simply the symplectic potential of a genuine toric K$\ddot{a}$hler structure underlying the toric GK structure, as described in Abreu-Guillemin theory.
\begin{theorem}On $\mathring{M}$, the two complex structures $J_0$ and $J_0'$ described above coincide. In particular, in the present setting, the K$\ddot{a}$hler form $\omega$ associated to Eq.~(\ref{kah}) coincides with $\Omega$. \label{coi}
\end{theorem}
\begin{proof}
In terms of admissible coordinates associated to $J_+\in DGK_\Omega^{\mathbb{T}}(M)$, $J_0'$, of course, should be of the form in Eq.~(\ref{avc}). So to prove the result, we only need to check that $J_0\partial_\mu=\phi_s\partial_\theta$. By definition,
\begin{eqnarray*}J_0\partial_{\theta}&=&-\mathbb{J}_1(\partial_{\theta}-b(\partial_{\theta}))=J_+(\partial_{\theta}-g^{-1}b(\partial_{\theta}))\\
&=&J_+[\textup{Id}-(J_++J_-)^{-1}(J_+-J_-)]\partial_\theta\\
&=&2J_+(J_++J_-)^{-1}J_-\partial_\theta\\
&=&2(J_-^{-1}+J_+^{-1})^{-1}\partial_\theta\\
&=&-(\frac{J_++J_-}{2})^{-1}\partial_\theta\\
&=&-(\phi_s)^{-1}\partial_\mu
\end{eqnarray*}
where Eq.~(\ref{sy}) is used. Thus, $J_0\partial_{\mu}=\phi_s \partial_\theta$ as required.

To see the two symplectic structures coincide, we recall from \cite{Wang} another interpretation of the K$\ddot{a}$hler structure $\omega$: One can identify $T\mathring{M}$ with $\mathbb{K}\oplus \mathbb{J}_1\mathbb{K}$ (see \S~2 and \S~3) via
\[X_j\mapsto X_j+\xi_j,\quad Y_j\mapsto Y_j,\]
where $Y_j=-\mathbb{J}_1(X_j+\xi_j)$. Then $J_0$ is nothing else but the restriction of $-\mathbb{J}_1$ on $\mathbb{K}\oplus \mathbb{J}_1\mathbb{K}$ and $\tilde{g}$ is the restriction of the generalized metric $-\mathbb{J}_1\mathbb{J}_2$ on $\mathbb{K}\oplus \mathbb{J}_1\mathbb{K}$. Consequently,
\begin{eqnarray*}\omega(X_j,X_k)&=&\omega(Y_j,Y_k)=\tilde{g}(J_0 X_j, X_k)
=(\mathbb{J}_1\mathbb{J}_2\mathbb{J}_1(X_j+\xi_j), X_k+\xi_k)\\
&=&-(\mathbb{J}_2(X_j+\xi_j), X_k+\xi_k)=\Omega(X_j,X_k)=0,
\end{eqnarray*}
where Eq.~(\ref{ga}) and the matrix form of $\mathbb{J}_2$ from \S~2 are used. Similarly,
\[\omega(X_j,Y_k)=(\mathbb{J}_1\mathbb{J}_2\mathbb{J}_1(X_j+\xi_j), Y_k)=-(\mathbb{J}_2(X_j+\xi_j), Y_k)=\Omega(X_j,Y_k).\]
The proof is then completed as expected.
\end{proof}
$J_0$ is essential to interpret properly the geometry of the underlying toric GK structure $\mathbb{J}_1$. In this respect, we have
\begin{theorem} \label{hol}On $\mathring{M}$, $\mathbb{J}_1$ is actually a B-transform of a GC structure $\mathbb{J}_\beta$ induced from a $J_0$-holomorphic Poisson structure $\beta:=-\frac{1}{4}(J_0\beta_1+\sqrt{-1}\beta_1)$.
\end{theorem}
\begin{proof}We divide the proof into three steps. Note that throughout the proof we will use the admissible coordinates associated with $J_+$.

\emph{Step 1}. Prove $\beta_1J_0^*=J_0\beta_1$ or equivalently $\beta_1$ is of type $(2,0)+(0,2)$ w.r.t. $J_0$.

Note that matrix forms of $\beta_1$ and $J_0$ are
\[\beta_1\sim\left(
                              \begin{array}{cc}
                                0 & -\phi_a(\phi_s)^{-1} \\
                                -(\phi_s)^{-1}\phi_a & 0 \\
                              \end{array}
                            \right),\quad J_0\sim \left(
            \begin{array}{cc}
              0 & -(\phi_s)^{-1} \\
              \phi_s & 0 \\
            \end{array}
          \right).\]
The claim then can be directly checked.

\emph{Step 2}. Prove $J_0\beta_1+\sqrt{-1}\beta_1$ is holomorphic w.r.t. $J_0$.

Let $Y_j=J_0\partial_{\theta_j}$. Firstly, one should note that $\{\partial_{\theta_j}-\sqrt{-1}Y_j\}$ is a $J_0$-holomorphic frame of $T\mathring{M}$. Let $du:=-\phi_sd\mu=-J_0^*d\theta$. Then $\{\theta_j, u_j\}$ is also a local coordinate system, and $Y_j=\partial_{u_j}$. Furthermore, $\{d\theta_j+\sqrt{-1}du_j\}$ is a $J_0$-holomorphic frame of $T^*\mathring{M}$  and
\[d\theta+\sqrt{-1}du=\left(
                \begin{array}{cc}
                  \textup{I} & -\sqrt{-1}\phi_s \\
                \end{array}
              \right)\left(
                       \begin{array}{c}
                         d\theta \\
                         d\mu \\
                       \end{array}
                     \right)
              .
\]
Secondly, the matrix form of $-4\beta$ is
\[J_0\beta_1+\sqrt{-1}\beta_1\sim \left(
                            \begin{array}{cc}
                              -\phi_a & -\sqrt{-1}\phi_a(\phi_s)^{-1} \\
                              -\sqrt{-1}(\phi_s)^{-1}\phi_a &(\phi_s)^{-1}\phi_a(\phi_s)^{-1}  \\
                            \end{array}
                          \right).
\]
Thus
\begin{eqnarray*}
-4\beta(d\theta+\sqrt{-1}du)&=&\left(
                \begin{array}{cc}
                  \textup{I} & -\sqrt{-1}\phi_s \\
                \end{array}
              \right)(J_0\beta_1+\sqrt{-1}\beta_1)\left(
                       \begin{array}{c}
                         d\theta \\
                         d\mu \\
                       \end{array}
                     \right)\\
                     &=&\left(
                \begin{array}{cc}
                  \textup{I} & -\sqrt{-1}\phi_s \\
                \end{array}
              \right)\left(
                            \begin{array}{cc}
                              -\phi_a & -\sqrt{-1}\phi_a(\phi_s)^{-1} \\
                              -\sqrt{-1}(\phi_s)^{-1}\phi_a &(\phi_s)^{-1}\phi_a(\phi_s)^{-1}  \\
                            \end{array}
                          \right)\left(
                                   \begin{array}{c}
                                     \partial_\theta \\
                                     \partial_\mu \\
                                   \end{array}
                                 \right)\\
 &=&\left(\begin{array}{cc}-2\phi_a & -2\sqrt{-1}\phi_a(\phi_s)^{-1} \\
\end{array} \right)\left(
                     \begin{array}{c}
                       \partial_\theta \\
                       \partial_\mu \\
                     \end{array}
                   \right)\\
                   &=&-2\phi_a(\partial_\theta+\sqrt{-1}(\phi_s)^{-1}\partial_\mu)\\
                   &=&-2\phi_a(\partial_\theta-\sqrt{-1}\partial_u).
\end{eqnarray*}
Since $\phi_a$ is a constant matrix, we have proved that $\beta$ is $J_0$-holomorphic. The fact that $\beta_1$ is a Poisson structure, together with the above two steps, implies that $\beta$ is a holomorphic Poisson structure relative to $J_0$ (or one can see this directly from the above local computation).

\emph{Step 3}. Prove that $\mathbb{J}_1$ is a B-transform of the GC structure $\mathbb{J}_\beta$.

We should find a 2-form $b_1$ such that
\[\frac{1}{2}\left(
  \begin{array}{cc}
    -J_+-J_-& \omega_+^{-1}-\omega_-^{-1} \\
    -\omega_++\omega_- & J_+^*+J_-^* \\
  \end{array}
\right)=\left(
          \begin{array}{cc}
            1 & 0 \\
            -b_1 & 1 \\
          \end{array}
        \right)\left(
                 \begin{array}{cc}
                   -J_0 & \beta_1 \\
                   0 & J_0^* \\
                 \end{array}
               \right)\left(
          \begin{array}{cc}
            1 & 0 \\
            b_1 & 1 \\
          \end{array}
        \right),
\]
which, in components, is equivalent to
\begin{equation}-\frac{1}{2}(J_++J_-)=-J_0+\beta_1b_1,\label{b1}\end{equation}
and
\begin{equation}-\frac{1}{2}(\omega_+-\omega_-)=b_1J_0-b_1\beta_1b_1+J_0^*b_1.\label{b2}\end{equation}
In terms of admissible coordinates $\{\theta_j, \mu_j\}$, we can choose
\[b_1\sim \left(
            \begin{array}{cc}
              (\phi^{-1})_a & 0 \\
              0 & 0 \\
            \end{array}
          \right).
\]
It is an elementary computation to see this $b_1$ does fulfill these equations (\ref{b1}) and (\ref{b2}).
\end{proof}
\emph{Remark}. Let $\phi_a=(C_{kj})$. From the above proof, in terms of the $J_0$-holomorphic frame $\partial_\theta-\sqrt{-1}\partial_u$, the holomorphic Poisson structure is
\begin{equation}\beta=\frac{1}{8}\sum_{k,j}C_{kj}(\partial_{\theta_j}-\sqrt{-1}\partial_{u_j})\wedge (\partial_{\theta_k}-\sqrt{-1}\partial_{u_k}).\label{Po}\end{equation}
Note that $\phi_a$ should be interpreted as the matrix form of a skew-symmetric bilinear function $c$ on $\mathfrak{t}^*$ relative to the basis $\{e_j\}$, i.e. \begin{equation}c=\frac{1}{2}\sum_{j,k}C_{kj}e_j\wedge e_k\end{equation} and that $\beta$ can be obtained from this expression by simply replacing $e_j$'s with the $J_0$-holomorphic part of $X_j$'s.\footnote{If we apply this operation to $c$ w.r.t. $J_\pm$, then we get the $J_\pm$-holomorphic Poisson structures underlying a GK structure which was first noted by N. Hitchin \cite{Hi}.} This fact will be used in \S~\ref{sub} and \S~\ref{GDC}. Another fact one should notice is that the 2-form $b_1$ appeared in the proof is actually a \emph{global} 2-form on $M$. Precisely, note that the 2-form $1/2\sum_{k,j}C_{kj}d\mu_j\wedge d\mu_k$ is globally defined on $M$ as $\mu_j$'s are. Therefore $b_1=b-1/2\sum_{k,j}C_{kj}d\mu_j\wedge d\mu_k$ is as well globally defined on $M$.

To conclude this subsection, let us have a look at the \emph{type} of $\mathbb{J}_1$ on $\mathring{M}$. At a point $x\in\mathring{M}$, this is the complex dimension transverse to the symplectic leaf of $\beta_1$ through $x$. $x$ is called regular if this number is constant around $x$. From Eq.~(\ref{Po}), we see that points in $\mathring{M}$ are all regular (actually $\mathring{M}$ is a $\mathbb{T}_\mathbb{C}$-orbit and the GC $\mathbb{T}_\mathbb{C}$-action preserves $\mathbb{J}_1$) and the common type is $n-\textup{rk}(\phi_a)$.
\subsection{Compactification}\label{Com}
Now we can start to tackle the subtle problem of compactification. It is the problem to determine whether a given $J_+\in DGK_{\Omega}^{\mathbb{T}}(\mathring{M})$ is actually the restriction of an element in $DGK_{\Omega}^{\mathbb{T}}(M)$ on $\mathring{M}$. The following Lemma~\ref{c1} and Thm.~\ref{com} can be viewed as a refined and strengthened version of the corresponding argument in \cite{Bou}.
\begin{lemma}\label{c1}$J_+\in GK_{\Omega}^{\mathbb{T}}(\mathring{M})$ is the restriction of an element in $GK_{\Omega}^{\mathbb{T}}(M)$ on $\mathring{M}$ if and only if all the canonically associated tensors $g, b$ and $(J_++J_-)^{-1}$ with $J_+$ in the metric splitting can be extended smoothly to $M$.
\end{lemma}
\begin{proof}It suffices to prove the sufficiency part of the lemma.

Due to the formulae $g=-\frac{1}{2}\Omega(J_++J_-)$ and $b=-\frac{1}{2}\Omega(J_+-J_-)$, if both $g$, $b$ can be extended smoothly to $M$, then both the tensors $J_++J_-$ and $J_+-J_-$ can as well be extended smoothly. This implies that $J_\pm$ are well-defined smooth tensors on $M$. A continuity argument makes it clear that $J_\pm$ are actually integrable complex structures on $M$.

By continuity, $g$ should be nonnegative-definite on $M\backslash \mathring{M}$. Since $\Omega=-2g(J_++J_-)^{-1}$, the smoothness of $(J_++J_-)^{-1}$ implies that $g$ must be non-degenerate on $M\backslash \mathring{M}$  and therefore positive-definite there.
\end{proof}

Take a reference element $I\in DGK_\Omega^\mathbb{T}(M)$ and let $ \{\theta_j, \mu_j\}$ be its associated admissible coordinates and $\psi$ its associated matrix. We assume the matrix form of $J_+\in DGK_\Omega^\mathbb{T}(\mathring{M})$ is anti-diagonal w.r.t. this admissible coordinate system and $\phi$ the corresponding matrix. Since $d\theta$, $\partial_\mu$ have no global meaning on $M$, we would like to use $I^*d\mu$ and $I\partial_\theta$ to replace them. Precisely this means we take $\{I^*d\mu, d\mu\}$ as a frame of $T^*\mathring{M}$ and $\{\partial_\theta, I\partial_\theta\}$ as a frame of $T\mathring{M}$. Then the tensors $g$, $b$ have the following forms respectively:
\[g=(I^*d\mu)^T\otimes[\psi(\phi^{-1})_s\psi^T(I^*d\mu)]+(d\mu)^T\otimes[\phi_s d\mu],\]
\[ b=(I^*d\mu)^T\otimes[\psi(\phi^{-1})_a\psi^T(I^*d\mu)] + (d\mu)^T\otimes[ \phi_a d\mu], \]
and the tensor $[(J_++J_-)/2]^{-1}$ has the form
\[[(J_++J_-)/2]^{-1}=(I^* d\mu)^T\otimes[\psi(\phi_s)^{-1}\psi(I \partial_\theta)]-(d\mu)^T\otimes [[(\phi^{-1})_s]^{-1}\partial_\theta].\]
Note that if $g_0$, $b_0$ are the corresponding metric and 2-form associated to $I$, and $g-g_0$, $b-b_0$, $[(J_++J_-)/2]^{-1}-[(I+I^\Omega)/2]^{-1}$ can be extended smoothly to $M$, then by Lemma~\ref{c1} $J_+\in DGK_\Omega^\mathbb{T}(M)$ as well. This observation implies:
\begin{theorem}\label{com}
$J_+\in DGK_\Omega^\mathbb{T}(\mathring{M})$ is the restriction of an element in $DGK_\Omega^\mathbb{T}(M)$ if the following conditions are satisfied:\\
i) $\phi_s-\psi_s$ and $\psi\phi^{-1}\psi^T-\psi^T$ can be extended smoothly to $\Delta$;\\
ii) $\phi^T(\phi_s)^{-1}\phi-\psi^T(\psi_s)^{-1}\psi$ and $\psi(\phi_s)^{-1}\psi-\psi(\psi_s)^{-1}\psi$ can be extended smoothly to $\Delta$.
\end{theorem}
\begin{proof}By Lemma~\ref{c1} and the above argument, we know that $J_+\in DGK_\Omega^\mathbb{T}(M)$ if the condition i) and the following condition hold:\\
ii') $[(\phi^{-1})_s]^{-1}-[(\psi^{-1})_s]^{-1}$ and $\psi(\phi_s)^{-1}\psi-\psi(\psi_s)^{-1}\psi$ can be extended smoothly to $\Delta$.\\
Note that $(\phi^{-1})_s=\phi^{-1}\phi_s(\phi^{-1})^T$ and thus
\[[(\phi^{-1})_s]^{-1}=\phi^T(\phi_s)^{-1}\phi.\]
So our theorem holds as required.
\end{proof}
Our conditions i) and ii) imply Boulanger's condition (C3) in \cite[Thm.~7]{Bou}. However, Boulanger's condition (C3), which resorts to the positive-definiteness of a bilinear form on the boundary of $\Delta$, seems a bit complicated and not very practical for use.

If the background $I$ is actually a K$\ddot{a}$hler structure, then $\psi=\psi_s$ and Thm.~\ref{com} can be simplified:
\begin{corollary}\label{c2}If $I\in DGK_\Omega^\mathbb{T}(M)$ is a K$\ddot{a}$hler structure, then $J_+\in DGK_\Omega^\mathbb{T}(\mathring{M})$ is the restriction of an element in $DGK_\Omega^\mathbb{T}(M)$ if the following conditions are satisfied:\\
i) $\phi_s-\psi$ and $\psi\phi^{-1}\psi-\psi$ can be extended smoothly to $\Delta$;\\
ii) $\phi^T(\phi_s)^{-1}\phi-\psi$ and $\psi(\phi_s)^{-1}\psi-\psi$ can be extended smoothly to $\Delta$.
\end{corollary}
As main applications of Thm.~\ref{com} and Corallary.~\ref{c2}, we prove two theorems.
\begin{theorem}Let $J_+\in DGK_\Omega^\mathbb{T}(M)$ and $\tau$ be the symplectic potential in Boulanger's theory (or in Thm.~\ref{GK} with $F_\zeta=0$). Then the same $\tau$ gives rise to an element $J_0\in K_\Omega^{\mathbb{T}}(M)$ in the sense of Abreu-Guillemin theory (or of Thm.~\ref{GK} with $F_\zeta=C=0$).
\end{theorem}
\begin{proof}We take $I=J_+$ as the reference toric GK structure. Let $\phi$ be the Hessian of $\tau$ and thus $\psi=\phi+C$ for some constant anti-symmetric matrix $C$. Then the matrix
\[\left(
    \begin{array}{cc}
      0 & -\phi^{-1} \\
      \phi & 0 \\
    \end{array}
  \right)
\]
relative to the admissible coordinates associated to $J_+$ obviously defines an element $J_0\in K_\Omega^{\mathbb{T}}(\mathring{M})$ (certainly this is the average complex structure investigated in the former subsection).

To see $J_0$ is actually defined globally on $M$, according to Thm.~\ref{com}, we only need to prove that
 both $(\phi+C)\phi^{-1}(\phi-C)-\phi$ and $\phi-(\phi-C)\phi^{-1}(\phi+C)$ admit a smooth extension to $\Delta$.

 Note that
 \begin{eqnarray*}(\phi+C)\phi^{-1}(\phi-C)-\phi&=&(\textup{I}+C\phi^{-1})(\phi-C)-\phi\\
 &=&\phi+C-C-C\phi^{-1}C-\phi\\
 &=&-C\phi^{-1}C\end{eqnarray*}
 and similarly
 \[\phi-(\phi-C)\phi^{-1}(\phi+C)=C\phi^{-1}C.\]
 Therefore we only need to check that $\phi^{-1}$ admits a smooth extension to $\Delta$.

 We know from Thm.~\ref{coi} that on $\mathring{M}$
 \[J_0\partial_{\theta_j}=Y_j=-\mathbb{J}_1(\partial_{\theta_j}-b(\partial_{\theta_j}))\]
 and meanwhile $J_0\partial_{\theta_j}=-\sum_k(\phi^{-1})^{kj}\partial_{\mu_k}$. Therefore, on $\mathring{M}$
 \[-d\mu_k(Y_j)=d\mu_k(\sum_l(\phi^{-1})^{lj}\partial_{\mu_l})=(\phi^{-1})^{kj}.\]
 Since both $\mu_k$ and $Y_j$ are globally defined on $M$, we know that these $-d\mu_k(Y_j)$'s are $\mathbb{T}$-invariant smooth functions on $M$, or equivalently smooth functions extending $(\phi^{-1})^{kj}$ to the whole of $\Delta$. This completes our proof.
\end{proof}
\emph{Remark}. In \cite{Bou}, Boulanger obtained a similar result in real dimension 4. He had observed that a K$\ddot{a}$hler structure associated to $J_+\in DGK_\Omega^\mathbb{T}(M)$ exists canonically in this dimension. In general he was not sure if $\phi^{-1}$ in the above proof is smooth on $\Delta$ or not. It is the global picture provided by \cite{Wang} that facilitates our analysis greatly. One should note the role played by the metric splitting (this was missing in \cite{Bou}) in our proof, justifying our introduction of the general formalism in \S~2 and \S~3.

The above theorem finishes our argument that the average complex structure $J_0$ in the former subsection is actually a global object on $M$. As a result, we can conclude that $\beta$ defined in the former subsection is a global holomorphic Poisson structure relative to $J_0$ and $\mathbb{J}_1$ has to be a B-transform of $\mathbb{J}_\beta$. Now in our present setting, we also come to a good understanding of the mysterious GC action of $\mathbb{T}_\mathbb{C}$.
\begin{corollary} For $J_+\in DGK_\Omega^\mathbb{T}(M)$, the associated GC action of $\mathbb{T}_\mathbb{C}$ covers the ordinary complexification of the $J_0$-holomorphic action of $\mathbb{T}$.
\end{corollary}
\begin{proof}
The result is obvious since at the infinitesimal level the extended $\mathfrak{t}_\mathbb{C}$-action covers the $\mathfrak{t}_\mathbb{C}$-action defined by
\[e_j\mapsto X_j,\quad \sqrt{-1}e_j\mapsto J_0X_j=Y_j,\]
which is the ordinary $J_0$-complexification of the $\mathfrak{t}$-action.
\end{proof}
In the other direction, we have
\begin{theorem}\label{c3}If $I\in K_\Omega^{\mathbb{T}}(M)$, $\tau$ is the symplectic potential of $I$ in Abreu-Guillemin theory, and $C$ is an $n\times n$ anti-symmetric constant matrix, then the pair $(\tau, C)$ gives rise to an element in $DGK_\Omega^\mathbb{T}(M)$ in the manner of Boulanger's theory.
\end{theorem}
\begin{proof}Now we take $I$ as the reference toric GK (actually K$\ddot{a}$hler) structure. Let $\psi$ be the Hessian of $\tau$ and $\phi=\psi+C$. Thus the matrix form
\[\left(
    \begin{array}{cc}
      0 & -\phi^{-1} \\
      \phi & 0 \\
    \end{array}
  \right)
\]
relative to the admissible coordinates of $I$ defines an element $J_+\in DGK_\Omega^\mathbb{T}(\mathring{M})$. To see $J_+$ actually lives in $DGK_\Omega^\mathbb{T}(M)$, by Corallary.~\ref{c2}, we have to prove both $\psi(\psi+C)^{-1}\psi-\psi$ and
$(\psi-C)\psi^{-1}(\psi+C)-\psi$ can be smoothly extended to $\Delta$.
Note that \[\psi(\psi+C)^{-1}\psi-\psi=-C(\psi+C)^{-1}\psi=-C(\textup{I}+\psi^{-1}C)^{-1}\]
and
\[(\psi-C)\psi^{-1}(\psi+C)-\psi=-C\psi^{-1}C.\]
 Note that $\psi^{-1}$ is smooth on $\Delta$, because $(\psi^{-1})^{kj}=\Omega(\partial_{\theta_j},I\partial_{\theta_k})$ is actually globally defined on $M$. Thus our proof amounts to checking that $\textup{I}+\psi^{-1}C$ is invertible on $\Delta$. Since $\psi^{-1}$ is positive-definite in $\mathring{\Delta}$, we can take the square root $\psi^{-1/2}$ and obtain
\[\det(\textup{I}+\psi^{-1}C)=\det(\psi^{-1/2})\cdot\det(\psi+C)\cdot \det(\psi^{-1/2})=\det(\textup{I}+\psi^{-1/2}C\psi^{-1/2}).\]
Notice that $\psi^{-1/2}C\psi^{-1/2}$ is anti-symmetric, and thus that in $\mathring{\Delta}$ we must have
\[\det(\textup{I}+\psi^{-1/2}C\psi^{-1/2})\geq 1.\]
By continuity, we have on $\Delta$ that $\det(\textup{I}+\psi^{-1}C)\geq 1$ and thus $\textup{I}+\psi^{-1}C$ is both smooth and invertible on $\Delta$. This completes our proof.
\end{proof}
\emph{Remark}. In \cite{Bou}, a similar result is obtained from the viewpoint of deformation theory of GK structures: given a compact toric K$\ddot{a}$hler manifold, the Hessian $\psi$ of the symplectic potential, together with an anti-symmetric matrix $C$ which is sufficiently small, produces a toric GK structure on $M$. Our theorem is a strengthened version of this result, \emph{without the smallness requirement of $C$}. In this respect, our theorem also goes beyond the scope of Goto's stability theorem.

Combining the two theorems together, we now can conclude that any element in $DGK_\Omega^\mathbb{T}(M)$ is actually obtained from a genuine toric K$\ddot{a}$hler structure by inputting an additional anti-symmetric matrix $C$ and any such prescription will give rise to an element in $DGK_\Omega^\mathbb{T}(M)$.

Recall that on a compact toric symplectic manifold $(M,\Omega, \mathbb{T}, \mu)$, the space of compatible toric K$\ddot{a}$hler structures is, modulo the action of $\mathbb{T}$-equivariant symplecomorphisms, a subspace $\mathcal{K}$ of continuous functions $\tau$ (the symplectic potential) on the moment polytope $\Delta$ satisfying the following two conditions \cite[Prop.~5]{Ap}:\\
i) The restriction of $\tau$ to any open face of $\Delta$ is a smooth strictly convex function;\\
ii) $\tau-\tau_0$ is smooth on $\Delta$, where $\tau_0$ is the standard symplectic potential associated to $\Delta$ (see Eq.~(\ref{formu})).
According to our theorems, we have the following proposition.
\begin{theorem}For a given compact toric symplectic manifold $(M,\Omega, \mathbb{T}, \mu)$, the space $\mathcal{DGK}$ of anti-diagonal toric GK structures of symplectic type modulo the action of $\mathbb{T}$-equivariant symplectomorphisms, is the product of $\mathcal{K}$ and the space $\mathcal{C}$ of $n\times n$ anti-symmetric constant matrices.
\end{theorem}
\begin{proof}If two elements in $DGK_\Omega^\mathbb{T}(M)$ have the same pair $(\tau, C)$ to characterize them, then the underlying K$\ddot{a}$hler structures $J_1, J_2$ are related by a $\mathbb{T}$-equivariant symplectomorphism $\Phi$. $\Phi$, as a holomorphic isomorphism between $(M, J_1)$ and $(M, J_2)$, also transforms the $J_1$-holomorphic Poisson structure associated to $C$ to the $J_2$-holomorphic Poisson structure associated to $C$, by simply replacing in Eq.~(\ref{Po}) the $J_1$-holomorphic vector fields $X_j^h$'s with their corresponding $J_2$-holomorphic vector fields.
\end{proof}
\subsection{GK submanifolds and type-jumping}\label{sub}
We have noted that points in $\mathring{M}$ are all regular for $\mathbb{J}_1\in DGK_\Omega^{\mathbb{T}}(M)$ and therefore irregular points necessarily lie in $M\backslash\mathring{M}$. On the other side, fixed points (always exist!) of the $\mathbb{T}$-action are of complex type because the $J_0$-holomorphic Poisson structure $\beta$ vanishes there. This subsection is then for completeness to sketch what happens to $M\backslash\mathring{M}$.

Since $\mathbb{J}_1$ is always a B-filed transform of $\mathbb{J}_\beta$, we can resort to some very general arguments concerning such structures in the existing literature. Especially Goto's work in \cite{Go} provides almost all the necessary material.

First, we review some basic notions.

\begin{definition}If $(M, J, \beta)$ is a $J$-holomorphic Poisson manifold with $\beta$ being its Poisson structure, a complex submanifold $X$ is called a Poisson submanifold if its defining ideal sheaf $I_X$ is a Poisson ideal of the structural sheaf $\mathcal{O}_M$, i.e., for any $f\in I_X$ and $g\in \mathcal{O}_M$ we have $\beta(df, dg)\in I_X$.
\end{definition}
Our toric GK manifold $M$ now carries a GC action of $\mathbb{T}_\mathbb{C}$, which covers an ordinary $J_0$-holomorphic action of $\mathbb{T}_\mathbb{C}$ on $M$. In this sense, $M$ is a complex toric manifold, i.e., an $n$-dimensional complex manifold carrying a holomorphic $\mathbb{T}^n_\mathbb{C}$-action having an open dense orbit. The $\mathbb{T}_\mathbb{C}$-action preserves the Poisson structure $\beta$, which takes a very special form like (\ref{Po}). This is precisely what Goto has considered in \cite[Example.~1.23]{Go}. As pointed out by Goto, for such a complex toric manifold, each toric submanifold is a Poisson submanifold of $\beta$. In our context, for each open face $F$ of $\Delta$, $M_F:=\mu^{-1}(\bar{F})$ is a complex toric submanifold and thus we know that $\mu^{-1}(\bar{F})$ is a Poisson submanifold of $\beta$.
\begin{definition} A submanifold $X$ of a GK manifold $(M, \mathbb{J}_1, \mathbb{J}_2)$ is called a GK submanifold if the pull-back $\iota^*L_1$ and $\iota^*L_2$ ($\iota$ is the inclusion map $\iota: X\hookrightarrow M$) of the corresponding complex Dirac structures $L_1$, $L_2$ are smooth and constitute a GK structure on $X$.
\end{definition}

Goto has observed in \cite{Go} that for a GK manifold $(M, \mathbb{J}_1, \mathbb{J}_2)$ where $\mathbb{J}_1$ is of the form $\mathbb{J}_\beta$, each Poisson submanifold of $\beta$ is automatically a GK submanifold.

Let $F$ be an open face of codimension $k$ of $\Delta$, defined in $(\mathbb{R}^n)^*$ by
\[(u_{j_l}, \mu)=\lambda_{j_l},\quad l=1,2,\cdots, k,\]
and $V_F$ the linear subspace of $(\mathbb{R}^n)^*$ singled out by $(u_{j_l},\mu)=0$, $l=1,2,\cdots, k$. Let $c_F$ be the restriction of $c=1/2\sum_{j,k}C_{kj}e_j\wedge e_k$ on $V_F$ and $r_F$ be the rank of $c_F$. Note that these $u_{j_l}\in \mathfrak{t}$ generate a subtorus $T_{0F}$ acting trivially on $M_F$. Let $T_F$ be the quotient of $\mathbb{T}$ by $T_{0F}$.
Then we have
\begin{theorem} Let $F$ be an open face of codimension $k$ of $\Delta$ as above. Then\\
i) for points in $\mu^{-1}(F)$, the type of $\mathbb{J}_1$ is $n-r_F$;\\
 ii) $M_F$ is a GK submanifold of $M$. Precisely, its GK structure $(\mathbb{J}_{1F}, \mathbb{J}_{2F})$ belongs to $DGK_{\Omega_F}^{\mathbb{T}_F}(M_F)$, where $\Omega_F$ is the restriction of $\Omega$ on $M_F$ and the type of $\mathbb{J}_{1F}$ in $\mu^{-1}(F)\subset M_F$ is $n-k-r_F$.
\end{theorem}
\begin{proof}Since the theorem is almost obvious, we only sketch a proof here. The pull-back of $L_2$ to a submanifold is a GC structure if and only if this submanifold is a symplectic submanifold and if this is the case, the pull-back of $L_2$ is again of symplectic type and the symplectic structure is precisely the pull-back of $\Omega$ (see \cite{Go}). Thus on $M_F$, $\mathbb{J}_{2F}$ is of symplectic type whose symplectic structure is $\Omega_F$, which is a toric symplectic structure (this is only a restatement in the GK setting of the classical result for toric symplectic manifolds).

Let $X_j^h$ be the $J_0$-holomorphic part of $X_j$ and $J_{0F}$ the restriction of $J_0$ on $M_F$. Then the $J_0$-holomorphic Poisson structure $\beta=1/2\sum_{j,l}C_{lj}X_j^h\wedge X_l^h$ and its restriction on $M_F$ is
\[\beta_F:=\frac{1}{2}\sum_{j,l}C_{lj}X_j^h|_F\wedge X_l^h|_F,\]
where $X_j^h|_F$ is the restriction of $X_j^h$ on $M_F$. Then $\mathbb{J}_{1F}$ is a B-transform of $\mathbb{J}_{\beta_F}$. Obviously all these structures on $M_F$ are $\mathbb{T}_F$-invariant and the $\mathbb{T}_F$-action is strong Hamiltonian. Thus $(\mathbb{J}_{1F}, \mathbb{J}_{2F})$ is really an element in $DGK_{\Omega_F}^{\mathbb{T}_F}(M_F)$.

We can choose the basis $\{e_j\}$ of $\mathfrak{t}$ such that $e_1,\cdots, e_k$ lie in the Lie algebra of $T_{0F}$. Then
\[\beta_F=\frac{1}{2}\sum_{j,l=k+1}^nC_{lj}X_j^h|_F\wedge X_l^h|_F,\]
which precisely corresponds to $c_F$ in the manner we have described at the end of \S~\ref{loc}. The remainder of this theorem follows directly from this observation.
\end{proof}
\subsection{An explicit example on $\mathbb{C}P^1\times \mathbb{C}P^1$}\label{ex}In this subsection, we construct a toric GK structure explicitly on $M=\mathbb{C}P^1\times \mathbb{C}P^1$ in the spirit of Thm.~\ref{c3}.
Let $M$ be equipped with the symplectic structure
\[\Omega=\frac{\sqrt{-1}}{2}\frac{dz_1\wedge d\bar{z}_1}{(1+|z_1|^2)^2}+\frac{\sqrt{-1}}{2}\frac{dz_2\wedge d\bar{z}_2}{(1+|z_2|^2)^2},\]
which is Hamiltonian relative to the standard $\mathbb{T}^2$-action:
\[(e^{\sqrt{-1}\theta_1},e^{\sqrt{-1}\theta_2})\cdot ([1:z_1], [1:z_2])=([1: e^{\sqrt{-1}\theta_1}z_1], [1:e^{\sqrt{-1}\theta_2}z_2)].\]
The infinitesimal action is given by
\[\partial_{\theta_j}=\sqrt{-1}(z_j\partial_{z_j}-\bar{z}_j\partial_{\bar{z}_j}),\quad j=1,2,\]
and the moment map for this toric symplectic manifold is chosen to be
\[\mu_j=\frac{|z_j|^2}{2(1+|z_j|^2)},\quad j=1,2.\]
Due to Guillemin's formula \cite{Gul} (see Eq.~(\ref{formu})), the symplectic potential of the standard toric K$\ddot{a}$hler structure in this case is
\[\tau=\frac{1}{2}\sum_{j=1}^2[\mu_j\ln \mu_j+(\frac{1}{2}-\mu_j)\ln(\frac{1}{2}-\mu_j)],\]
whose Hessian $\psi$ is
\[\left(
    \begin{array}{cc}
      \frac{1}{4\mu_1(1/2-\mu_1)} & 0 \\
      0 & \frac{1}{4\mu_2(1/2-\mu_2)} \\
    \end{array}
  \right).
\]

We would like to describe explicitly the many underlying geometric structures associated to an element of $DGK_\Omega^{\mathbb{T}^2}(M)$ whose related matrix $\phi$ is

\[\phi=\psi+C=\left(
    \begin{array}{cc}
      \frac{1}{4\mu_1(1/2-\mu_1)} & c \\
      -c & \frac{1}{4\mu_2(1/2-\mu_2)} \\
    \end{array}
  \right),\]
  where $c\neq 0$ is a real number. Note that
  \[\phi^{-1}=\frac{1}{\det \phi}\left(
                                   \begin{array}{cc}
                                     \frac{1}{4\mu_2(1/2-\mu_2)} & -c \\
                                     c &  \frac{1}{4\mu_1(1/2-\mu_1)} \\
                                   \end{array}
                                 \right),
  \]
  where $\det \phi=\frac{1}{16\mu_1(1/2-\mu_1)\mu_2(1/2-\mu_2)}+c^2$. For convenience, we introduce some notation:
  \[p:=\frac{1}{16\mu_1(1/2-\mu_1)\mu_2(1/2-\mu_2)},\quad \varrho_j=dz_j/z_j,\quad j=1,2.\]

   Obviously, in the present setting $\mathbb{J}_1$ should also be a B-transform of a symplectic structure $Q$ by a 2-form $b'$, at least on $\mathring{M}$. $Q$ should be the inverse of $\beta_1$ and just as how we obtain $b$ we have $b'=-\frac{1}{2}Q(J_++J_-)$. Then in terms of the admissible coordinates $\theta_j, \mu_j$, we can write down the several structures explicitly:
  \[g=\frac{4p}{p+c^2}\sum_{j=1}^2\mu_j(1/2-\mu_j)(d\theta_j)^2+\sum_{j=1}^2\frac{(d\mu_j)^2}{4\mu_j(1/2-\mu_j)},\]
    \[b=cd\mu_1\wedge d\mu_2-\frac{c}{p+c^2}d\theta_1\wedge d\theta_2,\]
    \[Q=\frac{1}{4c\mu_2(1/2-\mu_2)}d\theta_1\wedge d\mu_2-\frac{1}{4c\mu_1(1/2-\mu_1)}d\theta_2\wedge d\mu_1,\]
    and
    \[b'=\frac{p}{c(p+c^2)}d\theta_1\wedge d\theta_2-\frac{p}{c}d\mu_1\wedge d\mu_2.\]

Note that the pure spinors\footnote{In the paper, we haven't reviewed the pure spinor description of GC structures. For this see \cite{Gu0}.} of $\mathbb{J}_1$ and $\mathbb{J}_2$ are now clear. They are $e^{b'-\sqrt{-1}Q}$ and $e^{b-\sqrt{-1}\Omega}$. To compare with the standard K$\ddot{a}$hler structure on $\mathbb{C}P^1\times \mathbb{C}P^1$, let us consider the 2-form $\exp(b'-b-\sqrt{-1}Q)$ (A B-transform $e^b$ is used to modify $\mathbb{J}_2$ into $\mathbb{J}_\Omega$). Basically, we have
\[d\theta_j=-\frac{\sqrt{-1}}{2}(\varrho_j-\bar{\varrho}_j),\quad d\mu_j=\frac{|z_j|^2}{2(1+|z_j|^2)^2}(\varrho_j+\bar{\varrho}_j).\]
Substituting these into $b'-b-\sqrt{-1}Q$, we finally get
\[b'-b+\frac{c}{4p}(\varrho_1+\bar{\varrho}_1)\wedge(\varrho_2+\bar{\varrho}_2)-\sqrt{-1}Q=-\frac{\varrho_1\wedge \varrho_2}{c}. \]
Note that $\frac{c}{4p}(\varrho_1+\bar{\varrho}_1)\wedge(\varrho_2+\bar{\varrho}_2)$ is a real and global 2-form on $\mathbb{C}P^1\times \mathbb{C}P^1$ and the underlying holomorphic Poisson structure $\beta$ is $-cz_1z_2\partial_{z_1}\wedge \partial_{z_2}$, the inverse of $\frac{\varrho_1\wedge \varrho_2}{c}$. The type-jumping locus of $\mathbb{J}_1$ consists of four lines:
\[\{[1:0]\}\times \mathbb{C}P^1,\quad \{[0:1]\}\times \mathbb{C}P^1,\quad \mathbb{C}P^1\times\{[1:0]\},\quad \mathbb{C}P^1\times\{[0:1]\}.\]
Away from the four fixed points, the type-jumping locus is non-degenerate in the sense of Cavalcanti and Gualtieri in \cite{CG}.
 \section{Generalized Delzant construction}\label{GDC}
To construct a nontrivial GK structure in the manner of Thm.~\ref{c3}, one should first choose a background toric K$\ddot{a}$hler structure on $M$. For a fixed moment polytope $\Delta$, there is a standard choice called Delzant construction \cite{Del} \cite{Gul}. It is a concrete application of symplectic reduction (or more specifically K$\ddot{a}$hler reduction) to a torus action on $\mathbb{C}^d$. Let $(M_{\Delta}, \Omega, \mathbb{T}^n, \mu, J_0)$ be the toric K$\ddot{a}$hler manifold associated to $\Delta$ in Delzant's construction. If we choose this "standard" K$\ddot{a}$hler structure as the reference K$\ddot{a}$hler structure in Thm.~\ref{c3}, then inputting an additional anti-symmetric matrix $C$ will produce a toric GK structure on $M_\Delta$. This raises the natural question that\emph{ whether such a toric GK structure is the result of GK reduction of a toric GK structure on $\mathbb{C}^d$, just as $M_\Delta$ is the K$\ddot{a}$hler reduction of the standard K$\ddot{a}$hler structure on $\mathbb{C}^d$}. In this section, we mainly prove that this is really the case.
\subsection{The general mechanism}\label{gde}Given a Delzant polytope \footnote{For the definition of a Delzant polytope, see \cite{Gul}.} $\Delta$ in $(\mathbb{R}^n)^*$ defined by
\[l_j(x):=(u_j, x)\geq \lambda_j,\quad j=1,2,\cdots, d,\]
as found by Guillemin in \cite{Gul}, the standard symplectic potential
\begin{equation}\tau(x)=\frac{1}{2}\sum_{j=1}^dl_j(x)\ln l_j(x)\label{formu}\end{equation}
on $\mathring{\Delta}$ determines the standard K$\ddot{a}$hler structure on $M_\Delta$ in Delzant's construction. Now if an additional $n\times n$ anti-symmetric constant matrix $C$ is prescribed, then the pair $(\tau, C)$ generates a toric GK structure on $M_\Delta$ in the manner described in Thm.~\ref{c3}. We now prove that this slightly "twisted" version of the standard K$\ddot{a}$hler structure on $M_\Delta$ can be interpreted as reduced from a toric GK structure on $\mathbb{C}^d$.

Let us first describe the standard K$\ddot{a}$hler structure on $\mathbb{C}^d$ in the spirit of Abreu-Guillemin theory. $\mathbb{C}^d$ is equipped with the standard symplectic form
\[\Omega_0=\frac{\sqrt{-1}}{2}\sum_{j=1}^ddz_j\wedge d\bar{z}_j.\]
and the also standard action of a $d$-dimensional torus $\mathbb{T}^d$:
\[(e^{\sqrt{-1}\theta_1},\cdots,e^{\sqrt{-1}\theta_d})\cdot (z_1,\cdots, z_d)=(e^{\sqrt{-1}\theta_1}z_1,\cdots,e^{\sqrt{-1}\theta_d}z_d).\]
The infinitesimal action is generated by
\[\partial_{\theta_j}=\sqrt{-1}(z_j\partial_{z_j}-\bar{z}_j\partial_{\bar{z}_j}),\quad j=1,2,\cdots, d.\]
This action is Hamiltonian with a moment map $\nu: \mathbb{C}^d\rightarrow (\mathbb{R}^d)^*$, i.e.,
\[\nu(z_1,\cdots, z_d)=\frac{1}{2}(|z_1|^2+2\lambda_1, |z_2|^2+2\lambda_2,\cdots, |z_d|^2+2\lambda_d),\]
or
\[\nu_j=\frac{1}{2}|z_j|^2+\lambda_j,\quad j=1,2,\cdots, d.\]
In terms of admissible coordinates $\theta, \nu$, the metric on $\mathbb{C}^d$ is of the following form:
\[g_0=\sum_{j=1}^d(|z_j|^2(d\theta_j)^2+\frac{(d\nu_j)^2}{|z_j|^2}).\]
Thus the canonical K$\ddot{a}$hler structure is described by the diagonal matrix
\[\psi_0=\textup{Diag}\{1/|z_1|^2, 1/|z_2|^2,\cdots, 1/|z_d|^2\}.\]
and the corresponding symplectic potential is
\[\tau_0=\frac{1}{2}\sum_{j=1}^d(\nu_j-\lambda_j)\ln(\nu_j-\lambda_j).\]

To see how the anti-symmetric $n\times n$ matrix $C$ is related to structures on $\mathbb{C}^d$, note that in Delzant's construction we have the linear map $\varsigma: \mathbb{R}^d\rightarrow \mathbb{R}^n$, $e_j\mapsto u_j$, where $\{e_j\}$ is the standard basis of $\mathbb{R}^d$. Let $\mathfrak{n}$ be the kernel of $\varsigma$. Then we have the short exact sequence
\begin{equation}0\longrightarrow \mathfrak{n}\stackrel{\iota}\longrightarrow \mathbb{R}^d \stackrel{\varsigma}\longrightarrow \mathbb{R}^n \longrightarrow0,\label{seq}\end{equation}
where each term should be interpreted as the Lie algebra of the corresponding torus and $\iota$ is the natural inclusion map. We have an induced map $\varsigma_\wedge: \wedge^2 \mathbb{R}^d\rightarrow \wedge^2 \mathbb{R}^n$. Intrinsically understood, $C$ is a skew-symmetric bilinear function on $(\mathbb{R}^n)^*$ in which the moment map $\mu$ on $M_\Delta$ takes values. This means $C$ lives in \[\wedge^2 [(\mathbb{R}^n)^*]^*\cong \wedge^2 \mathbb{R}^n\cong \wedge^2 \mathfrak{t}^n.\] Let $C_0\in \wedge^2 \mathbb{R}^d$ such that $\varsigma_\wedge(C_0)=C$. Then $C_0$ is a skew-symmetric bilinear function on $(\mathbb{R}^d)^*$ in which the moment map $\nu$ on $\mathbb{C}^d$ takes values or $C_0\in \wedge^2 \mathfrak{t}^d$.

To summarize, we have
\begin{lemma}The pair $(\tau_0, C_0)$ determines a toric GK structure $\mathbb{J}_{10}\in DGK_{\Omega_0}^{\mathbb{T}^d}(\mathbb{C}^d)$ in the manner of Thm.~\ref{c3}.
\end{lemma}
\begin{proof}
The pair $(\tau_0, C_0)$ certainly defines a toric GK structure in $DGK_{\Omega_0}^{\mathbb{T}^d}((\mathbb{C}^*)^d)$. To see it extends smoothly on the whole of $\mathbb{C}^d$. We can go along the same line in the proof of Thm.~\ref{c3}, and thus omit the details.
\end{proof}
 It is of some value to have a close look at the GK structure on $\mathbb{C}^d$ before going on. In the present context, the average complex structure $I_0$ is actually the standard one on $\mathbb{C}^d$ and thus the $I_0$-holomorphic action of $\mathbb{T}^d$ can be complexified canonically. Even more, this $\mathbb{T}^d$-action can be generalized complexified: Although in \cite[Thm.~5.7]{Wang}, the author only proved that a strong Hamiltonian action on a compact GK manifold can be generalized complexified, this still holds for our present non-compact situation because we already know what the underlying $\mathbb{T}^d_\mathbb{C}$-action on $\mathbb{C}^d$ is and the proof of \cite[Thm.~5.7]{Wang} continues to be valid without any essential modification. Additionally, if $C_0=1/2\sum_{k,j=1}^d C_{0kj}e_j\wedge e_k $, then the underlying $I_0$-holomorphic Poisson structure $\beta_0$ is
 \[\beta_0=-\frac{1}{2}\sum_{k,j=1}^d C_{0kj}z_jz_k\partial_{z_j}\wedge \partial_{z_k},\]
 which is quadratic.

Note that the short exact sequence (\ref{seq}) of Lie algebras lifts to the level of Lie groups:
\begin{equation}0\longrightarrow N \longrightarrow \mathbb{T}^d \longrightarrow \mathbb{T}^n \longrightarrow0.\label{seq2}\end{equation}
\begin{theorem}The toric GK structure on $M_\Delta$ associated to $(\tau, C)$ is the GK reduction of the toric GK structure on $\mathbb{C}^d$ associated to $(\tau_0, C_0)$ by the strong Hamiltonian action of $N$.
\end{theorem}
 \begin{proof}
 First we should notice that the present situation does fit in well with the general formalism developed in \cite{LT} \cite{Wang}. So we do have a GK quotient by the action of $N$. Additionally, the story is rather trivial from the symplectic side--it is essentially the classical symplectic reduction. The point now is to show the quotient GK structure does lie in $DGK_\Omega^{\mathbb{T}^n}(M_\Delta)$ and is parameterized by the pair $(\tau, C)$. We divide the proof into several steps:

 \emph{Step 1.} The reduced GK structure $\mathbb{J}_1\in GK_\Omega^{\mathbb{T}^n}(M_\Delta)$.

 One can choose a splitting of the short exact sequence (\ref{seq2}) such that $\mathbb{T}^d=N\times \mathbb{T}^n$. This is always possible, e.g. one can choose the stabilizer of $x\in Z:=\nu^{-1}(0)$ such that $\mu([x])$ is a vertex of $\Delta$ where $[x]$ is the image of $x$ in the quotient $Z/N$. The action of $\mathbb{T}^n$ on $\mathbb{C}^d$ commutes with that of $N$ and the residual $\mathbb{T}^n$ action on $M_\Delta$ should preserve the many reduced structures. The reduced GK structure is obviously of symplectic type and acquires the moment map $\mu$ in the usual way. This shows $\mathbb{J}_1\in GK_\Omega^{\mathbb{T}^n}(M_\Delta)$.

 \emph{Step 2.} The reduced GK structure $\mathbb{J}_1\in DGK_\Omega^{\mathbb{T}^n}(M_\Delta)$.

 According to Thm.~\ref{cha}, we only need to show the residual $\mathbb{T}^n$-action is strong Hamiltonian. We should change our viewpoint slightly to look at the GIT quotient of $\mathbb{J}_{10}$. The open set $N_\mathbb{C}Z\subset \mathbb{C}^d$ inherits its natural GC structure $\mathbb{J}_{10}|_{N_\mathbb{C}Z}$ by restriction and $N_\mathbb{C}$ acts on it in a GC fashion. Then $\mathbb{J}_1$ on $M_\Delta$ is precisely the reduction of $\mathbb{J}_{10}|_{N_\mathbb{C}Z}$ by the GC $N_\mathbb{C}$-action and on $M_\Delta$
 \[\beta_1(d\mu_j, d\mu_k)=(\mathbb{J}_1[X_j+\xi_j], [X_k+\xi_k])=(\mathbb{J}_{10}(X_j+\xi_j), X_k+\xi_k)=0,\]
 where $\{X_j+\xi_j\}$ denotes the infinitesimal extended action on $N_\mathbb{C}Z$ of a basis $f_j\in \mathfrak{t}^n$ and $[X_j+\xi_j]$ its image under the quotient map. The last equality is due to the fact that the action of $\mathbb{T}^d$ on $\mathbb{C}^d$ is strong Hamiltonian. This proves that the residual $\mathbb{T}^n$-action on $M_\Delta$ is strong Hamiltonian.

 \emph{Step 3}. The average complex structure $J_0$ associated to $\mathbb{J}_1$ is parameterized by $\tau$.

 This is fairly easy if we again pay attention to the GIT quotient of $\mathbb{J}_{10}$. Note that $J_0$ restricted on $\mathring{M}_\Delta$ is actually obtained from the GC action of $\mathbb{T}^n_\mathbb{C}$ while this action stems from the quotient action of $\mathbb{T}^d_\mathbb{C}$ by $N_\mathbb{C}$. We have noticed that the GC action of $\mathbb{T}^d_\mathbb{C}$ on $\mathbb{C}^d$ covers the standard $\mathbb{T}_\mathbb{C}^d$-action. That's to say $J_0$ restricted on $\mathring{M}_\Delta$ is essentially the standard complex structure on $\mathbb{T}^d_\mathbb{C}/N_\mathbb{C}$, which is clearly, due to the famous Kempf-Ness theorem of classical GIT, the one determined by $\tau$ from the symplectic side. 

  \emph{Step 4}. The underlying $J_0$-holomorphic Poisson structure $\beta$ associated to $\mathbb{J}_1$ is parameterized by $C$.

  Note that $\beta$ can be extracted from the formula \[\beta(df ,dg)=-\frac{\sqrt{-1}}{2}\beta_1(df, dg)=-\frac{\sqrt{-1}}{2}(\mathbb{J}_1 df, dg),\] where $f, g$ are local $J_0$-holomorphic functions on $M_\Delta$. Let $q$ be the quotient map from $N_\mathbb{C}Z$ to $M_\Delta$. Then $q^*(f), q^*(g)$ are $I_0$-holomorphic functions on $N_\mathbb{C}Z$ and
  \[(\mathbb{J}_1 df, dg)=(\mathbb{J}_{10}d q^*(f), dq^*(g))=2\sqrt{-1}\beta_0(dq^*(f), dq^*(g)).\]
  That's to say $\beta=q_*(\beta_0)$. This implies that $\beta$ is precisely the $J_0$-holomorphic Poisson structure corresponding to $C$.
  \end{proof}
  The above theorem has an interesting implication. Usually, the K$\ddot{a}$hler structure on $M_\Delta$ obtained via Delzant's construction is regarded as canonical, but from the viewpoint of our theorem this is however not that canonical. If one wants to obtain only the toric symplectic structure on $M_\Delta$, there is still much freedom to decide what to start with. Even if one wants to obtain the canonical K$\ddot{a}$hler structure on $M_\Delta$, he can still use a nontrivial toric GK structure on $\mathbb{C}^d$ from the very beginning, simply requiring that $C_0$ lie in the kernel of $\varsigma_\wedge$.

\begin{example}We can use the generalized Delzant construction to produce a GK structure on $\mathbb{C}P^2$. In the present situation $d=3$ and $\mathbb{T}^3$ acts on $\mathbb{C}^3$ in the standard fashion. The action of $N=S^1$ is simply by scaling.
We choose
\[C_0=\left(
    \begin{array}{ccc}
      0 & c_1 & c_2 \\
      -c_1 & 0 & c_3 \\
      -c_2 & -c_3 & 0 \\
    \end{array}
  \right),
\]
or
\[C_0=c_1e_1\wedge e_2+c_2e_1\wedge e_3+c_3e_2\wedge e_3.\]
Note that in this case, $\varsigma(e_1)=(1,0)$, $\varsigma(e_2)=(0,1)$, $\varsigma(e_3)=(-1,-1)$ and thus $\varsigma(e_1+e_2+e_3)=0$. Therefore,
\[\varsigma_\wedge(C_0)=(c_1-c_2+c_3)\varsigma(e_1)\wedge\varsigma(e_2).\]
It is this combination $c_1-c_2+c_3$ that will finally matter for the reduced structure on $\mathbb{C}P^2$. In particular, if $c_1+c_3=c_2$ while $c_1, c_2, c_3$ are not all zero, the reduced structure is the standard K$\ddot{a}$hler structure on $\mathbb{C}P^2$, although the initial structure on $\mathbb{C}^3$ is not its standard one.
\end{example}
In \cite{LT}, a similar reduction procedure was applied to $\mathbb{C}^d$ to get nontrivial GK structures of symplectic type. The starting point was to fix $\Omega_0$ and meanwhile deform the standard complex structure on $\mathbb{C}^d$ in a compatible way. GK reduction then assures that the quotient be a nontrivial GK structure of symplectic type. The generalized Delzant construction introduced above seems conceptually much easier. However, it can not be viewed as a practical method to produce toric GK structures for all such structures on $M_\Delta$ can be directly obtained by using Guillemin's formula (\ref{formu}) for the symplectic potential and an additional anti-symmetric matrix. Generalized Delzant construction can really be used for other purposes: In the next subsection, it will be used to construct non-abelian examples of strong Hamiltonian actions. Another possible use was suggested by \cite{Wang}: $N_\mathbb{C}Z$ in the construction is a generalized holomorphic principal $N_\mathbb{C}$-bundle, and it can be used to produce $\mathbb{J}_1$-generalized holomorphic vector bundles (actually holomorphic Poisson modules in this case, see \cite{Gu2}) on $M_\Delta$. We will investigate this possible use elsewhere.

\subsection{Examples of non-abelian strong Hamiltonian action}\label{non}
Up to now, we have only been dealing with strong Hamiltonian actions of tori. Actually, in \cite{Wang} only one example of non-abelian strong Hamiltonian action was given. In this subsection, we construct more examples by using generalized Delzant construction introduced in \S~\ref{gde}. The final result, more or less, is similar to the example provided in \cite{Wang}, but conceptually simpler and more practical.

The idea goes as follows: Starting with an $(n+1)\times (n+1)$ anti-symmetric matrix $C$, by generalized Delzant construction we can construct a toric GK structure $(\mathbb{J}_1, \mathbb{J}_2)$ of symplectic type on $\mathbb{C}P^n$ whose symplectic structure is exactly the standard Fubini-Study form $\Omega$. However, $\Omega$ has much more symmetries than just $\mathbb{T}^n$. Precisely, $U(n+1)$ acts naturally on $\mathbb{C}P^n$ in a Hamiltonian fashion and even the underlying average complex structure $J_0$ is preserved by this action, but in general it will not preserve the $J_0$-holomorphic Poisson structure $\beta$. However, it may be possible that by choosing $C$ and a non-abelian subgroup $G\subset U(n+1)$ carefully we can finally get the examples we want.

Let $0<k<n$ and $C=1/2\sum_{j,l=1}^{k-1}c_{lj}e_j\wedge e_l$, which gives rise to a toric GK structure on $\mathbb{C}^{n+1}$. We have the decomposition $\mathbb{C}^{n+1}=\mathbb{C}^k\oplus \mathbb{C}^{n+1-k}$ by writing a vector $v\in \mathbb{C}^{n+1}$ as the sum of its first $k$ components and latter $n+1-k$ components. The subgroup $U(n+1-k)\subset U(n+1)$ acts on $\mathbb{C}^{n+1}$ by fixing $\mathbb{C}^k$ and acting on $\mathbb{C}^{n+1-k}$ in the standard way. Obviously this $U(n+1-k)$-action on $\mathbb{C}^{n+1}$ commutes with the diagonal $S^1$-action on $\mathbb{C}^{n+1}$ by scaling and preserves the GK structure on $\mathbb{C}^{n+1}$. By generalized Delzant construction, we obtain a GK structure on $\mathbb{C}P^n$. The $U(n+1-k)$-action also descends to $\mathbb{C}P^n$, preserving its GK structure. Moreover, this $U(n+1-k)$-action is Hamiltonian w.r.t. the Fubini-Study form $\Omega$ on $\mathbb{C}P^n$.
\begin{proposition}The $U(n+1-k)$-action on the GK manifold $\mathbb{C}P^n$ described as above is strong Hamiltonian.
\end{proposition}
\begin{proof}It suffices to prove the result in an affine coordinate system. Let $[z_0,z_1,\cdots,z_n]$ be the homogeneous coordinates of $\mathbb{C}P^n$. In the chart $\{z_0\neq 0\}$ we use the affine coordinates
$w_j=z_j/z_0$, $j=1,2,\cdots, n$.

It is well-known that in terms of homogeneous coordinates the moment map of the $U(n+1)$-action is (see for example \cite{Kir})
\[\nu([p])=\frac{p^*p}{2\|p\|^2},\quad p=(z_0,z_1,\cdots, z_n),\]
where $p^*$ is the complex conjugate transpose of $p$ and we have used the trace to identify $\sqrt{-1}\mathfrak{u}(n+1)$ with its dual. Thus its restriction on $\mathfrak{u}(n+1-k)$ in components is
\[\mu_{lj}^r([p])=\frac{\bar{z}_jz_l+z_j\bar{z}_l}{4\|p\|^2},\quad \mu_{lj}^i([p])=\frac{\bar{z}_jz_l-z_j\bar{z}_l}{4\sqrt{-1}\|p\|^2}, \quad l, j=k,k+1,\cdots, n,\]
or in affine coordinates
\[\mu_{lj}^r([p])=\frac{\bar{w}_jw_l+w_j\bar{w}_l}{4(1+|w|^2)},\quad \mu_{lj}^i([p])=\frac{\bar{w}_jw_l-w_j\bar{w}_l}{4\sqrt{-1}(1+|w|^2)}\quad l, j=k, k+1,\cdots, n,\]
where $|w|^2=\sum_{j=1}^n|w_j|^2$ and the superscripts $r$ and $i$ denote the real and imaginary parts respectively.

Since on $\mathbb{C}^{n+1}$ the holomorphic Poisson structure $\beta_0$ is
\[\beta_0=-\frac{1}{2}\sum_{j,l=1}^{k-1}c_{lj}z_jz_l\partial_{z_j}\wedge \partial_{z_l},\]
its reduced version on $\mathbb{C}P^n$ is
\[\beta=-\frac{1}{2}\sum_{j,l=1}^{k-1}c_{lj}w_jw_l\partial_{w_j}\wedge \partial_{w_l}.\]

Note that either $\mu_{lj}^r$ or $\mu_{lj}^i$ is of the form $h/(1+|w|^2)$, where $h$ is a Casimir function of $\beta$. Then the $\beta$-Poisson bracket of any two components of the moment map $\mu$ has the form
\[h'\{\frac{1}{1+|w|^2}, \frac{1}{1+|w|^2}\}_\beta=0,\]
where $h'$ is another Casimir function. This is enough for deriving that the $\beta_1$-Poisson bracket of any two components of $\mu$ also vanishes. Our proof is thus completed.
\end{proof}

\section{Appendix}We collect some facts concerning matrices here. They are elementary but frequently (maybe implicitly) used in the main text of this paper. For a matrix $A$, let $A^T$ be its transpose and $A_s$, $A_a$ its symmetric and anti-symmetric parts respectively. In the following, let $A$ be an $n\times n$ invertible matrix and $B$ its inverse.

\textbf{Fact I}. \[A_sB_s+A_aB_a=B_sA_s+B_aA_a=\textup{I},\]
\[A_sB_a+A_aB_s=B_sA_a+B_aA_s=0.\]

\textbf{Fact II.} $A_s$ is invertible if and only if $B_s$ is invertible; in particular, if $A_s$ is positive-definite, then so is $B_s$.

\textbf{Fact III}. $AB_sA^T=A_s$ and $AB_aA^T=-A_a$. In particular, we have \[B^T(B_s)^{-1}B=(A_s)^{-1}.\]
\section*{Acknowledgemencts}
This study is supported by the Natural Science Foundation of Jiangsu Province (BK20150797). The manuscript is prepared during the author's stay in the Department of Mathematics at the University of Toronto and this stay is funded by the China Scholarship Council (201806715027). The author also thanks Professor Marco Gualtieri for his invitation and hospitality.


\begin{thebibliography}{00}
\bibitem{Ab}
M. Abreu, K$\ddot{a}$hler geometry of toric manifolds in symplectic coordinates, in ¡¯Symplectic and Contact Topology: Interactions and Perspectives¡¯ (eds. Y.Eliashberg, B. Khesin and F. Lalonde), Fields Institute Communications 35, American Mathematical Society, 2003.
\bibitem{Ap}
V. Apostolov, The K$\ddot{a}$hler geometry of toric manifolds, Lecture notes, 2017. www.cirget.uqam.ca/~apostolo/papers/toric-lecture-notes.pdf
\bibitem{At} M. F. Atiyah, Convexity and commuting Hamiltonians, Bulletin of the London Mathematical Society, Vol 14, Issue 1, 1-15, 1982.
\bibitem{Bou}
L. Boulanger, Toric generalized K$\ddot{a}$hler structures, arXiv:1509.06785v2.
\bibitem{BCG1}
 H. Bursztyn, G.R. Cavalcanti, and M. Gualtieri, Reduction of Courant algebroids and generalized complex structures, Adv. Math. 211, no. 2, 726-765, 2007.
 \bibitem{CG}
 G.R. Cavalcanti and M. Gualtieri, A surgery for generalized complex structures on 4-manifolds, J. Differential Geom. 76, no. 1, 35-43, 2007.
\bibitem{Del}
T. Delzant, Hamiltoniens p¨¦riodiques et image convexe de l¡¯application moment, Bull. Soc. Math. France 116, 315-339, 1998.
\bibitem{Du}
J. L. van der Leer Dur¨¢n, Blow-ups in generalized K$\ddot{a}$hler geometry, Commun. Math. Phys. Vol. 357, 1133-1156, 2018.
\bibitem{En}
N. Enrietti, A. Fino, G. Grantcharov, Tamed symplectic forms and generalized geometry, J. Geom. Phys. 71, 103-116, 2013.
\bibitem{Go1}
R. Goto, Deformations of generalized complex and generalized K$\ddot{a}$hler structures, J. Differential Geom. 84, no. 3, 525-560, 2010.
\bibitem{Go}
R. Goto, Poisson structures and generalized K$\ddot{a}$hler manifolds, J. Math. Soc. Japan Vol. 61, No. 1, 107-132, 2009.
 \bibitem{Gu00}
M. Gualtieri, Generalized complex geometry, PhD thesis, Oxford University, 2003. arXiv: math./0401221.
 \bibitem{Gu0}
M. Gualtieri, Generalized complex geometry, Ann. of Math, 174: pp. 75-123, 2011.
\bibitem{Gu1}
M. Gualtieri, Generalized K$\ddot{a}$hler geometry, Commun. Math. Phys. 331, 297-331, 2014.
\bibitem{Gu2}
M. Gualtieri, Branes on Poisson varieties, in The many facets of geometry, 368-394. Oxford Univ. Press, Oxford, 2010.
\bibitem{Gul}
V. Guillemin, K$\ddot{a}$hler structures on toric varieties, J. Differential Geom. 40, 285-309, 1994.
 \bibitem{GS}
 V. Guillemin, S. Sternberg, Convexity properties of the moment mapping, Invent. Math. 67, 491-513, 1982.
 \bibitem{Hi}
 N. Hitchin, Instantons, Poisson structures and generalized Kaehler geometry, Comm. Math. Phys. 265, no.1, 131-164, 2006. arXiv:math/0503432v1.
 \bibitem{Kir}
  F. Kirwan, Cohomology of quotientsin symplectic and algebraic geometry, Mathematical Notes, 31, Princeton, NJ, 1984.
  \bibitem{LT}
  Y. Lin, and S. Tolman, Symmetries in generalized K$\ddot{a}$hler geometry, Commun. Math. Phys. 268, 199-222, 2006. arXiv:math./0509069
\bibitem{Wang}
Y. Wang, The GIT aspect of generalized K$\ddot{a}$hler reduction. I., arxiv:1803.01178v2.
 \end{thebibliography}
\end{document}